\documentclass[13pt, reqno]{amsart}

\usepackage{graphicx}
\usepackage[normalem]{ulem}
\usepackage{mathrsfs}
\usepackage{amssymb}
\usepackage{amsfonts, mathtools}
\usepackage{amssymb,amsmath,amsthm,tikz-cd, tikz}
\usepackage{stix}

\usepackage{hyperref}
\usepackage{palatino}
\usepackage{wrapfig}
\usepackage{subcaption}

\usetikzlibrary{patterns, matrix,shapes,arrows,calc,topaths,intersections,positioning,decorations.pathreplacing,decorations.pathmorphing,fit, 3d}

\newif\ifpaper
\paperfalse
\usepackage{color}

\title {Variation of GIT and Variation of Lagrangian Skeletons I:  Flip and Flop}
\author{Peng Zhou}
\date{\today}

\usepackage{color}

\renewcommand{\ss}{\subsection}
\newcommand{\sss}{\subsubsection}

\DeclareMathOperator{\cone}{cone}

\DeclareMathOperator{\Coh}{Coh}

\DeclareMathOperator{\End}{End}

\DeclareMathOperator{\Fuk}{Fuk}

\DeclareMathOperator{\Hom}{Hom}

\newcommand{\uhom}{{\hcal  om}}

\DeclareMathOperator{\Perf}{Perf}

\DeclareMathOperator{\spec}{Spec}

\DeclareMathOperator{\Skel}{Skel}

\DeclareMathOperator{\Vect}{Vect}

\newcommand{\NN}{\mathsf{N}}
\newcommand{\MM}{\mathsf{M}}

\newcommand{\dm}{\diamondsuit}

\newcommand{\z}{\text}
\newcommand{\qo}{/\,}

\newcommand{\pa}{\partial}

\newcommand{\la}{\langle}
\newcommand{\ra}{\rangle}
\newcommand{\ot}{\otimes}

\renewcommand{\b}{\bar}

\newcommand{\wt}{\widetilde}

\def\bfone{{\bf 1}}
\def\bfx{{\bf x}}

\newcommand{\acal}{\mathcal{A}}
\newcommand{\bcal}{\mathcal{B}}
\newcommand{\ccal}{\mathcal{C}}
\newcommand{\dcal}{\mathcal{D}}

\newcommand{\fcal}{\mathcal{F}}
\newcommand{\gcal}{\mathcal{G}}
\newcommand{\hcal}{\mathcal{H}}
\newcommand{\ical}{\mathcal{I}}
\newcommand{\jcal}{\mathcal{J}}

\newcommand{\lcal}{\mathcal{L}}

\newcommand{\ocal}{\mathcal{O}}
\newcommand{\pcal}{\mathcal{P}}

\newcommand{\tcal}{\mathcal{T}}

\newcommand{\wcal}{\mathcal{W}}
\newcommand{\xcal}{\mathcal{X}}
\newcommand{\ycal}{\mathcal{Y}}

\newcommand{\C}{\mathbb{C}}

\newcommand{\R}{\mathbb{R}}

\newcommand{\Z}{\mathbb{Z}}

\newcommand{\T}{\mathbb{T}}
\renewcommand{\P}{\mathbb{P}}

\newcommand{\xto}{\xrightarrow}

\newcommand{\RM}{\backslash}

\newcommand{\into}{\hookrightarrow}

\newcommand{\onto}{\twoheadrightarrow}

\newcommand{\bea}{\begin{eqnarray*} }
\newcommand{\eea}{\end{eqnarray*} }
\newcommand{\be}{\begin{equation} }
\newcommand{\ee}{\end{equation} }
\newcommand{\bp}{\begin{proposition}}
\newcommand{\ep}{\end{proposition}}
\newcommand{\bt}{\begin{theo}}
\newcommand{\et}{\end{theo}}

\newcommand{\btu}{\begin{theou}}
\newcommand{\etu}{\end{theou}}
\newcommand{\bpf}{\begin{proof}}
\newcommand{\epf}{\end{proof}}
\newcommand{\bl}{\begin{lemma}}
\newcommand{\el}{\end{lemma}}
\newcommand{\bc}{\begin{corollary}}
\newcommand{\ec}{\end{corollary}}
\newcommand{\bd}{\begin{definition}}
\newcommand{\ed}{\end{definition}}
\newcommand{\bex}{\begin{example}}
\newcommand{\eex}{\end{example}}
\newcommand{\bA}{\left(\begin{array}}
\newcommand{\eA}{\end{array}\right)}
\newcommand{\bma}{\begin{bmatrix}}
\newcommand{\ema}{\end{bmatrix}}
\newcommand{\bcd}{\begin{tikzcd}}
\newcommand{\ecd}{\end{tikzcd}}
\newcommand{\bcs}{\begin{cases}}
\newcommand{\ecs}{\end{cases}}
\newcommand{\bee}{\begin{eqnarray} }
\newcommand{\eee}{\end{eqnarray} }
\newcommand{\brem}{\begin{remark}}
\newcommand{\erem}{\end{remark}}
\newcommand{\bnum}{\begin{enumerate}}
\newcommand{\enum}{\end{enumerate}}

\newtheorem*{theou}{Main Theorem}
\newtheorem{theo}{Theorem}[section]
\newtheorem{lemma}[theo]{Lemma}
\newtheorem{corollary}[theo]{Corollary}
\newtheorem{proposition}[theo]{Proposition}

\newtheorem{definition}[theo]{Definition}
\newtheorem{remark}[theo]{Remark}

\newtheorem{examplex}[theo]{Example}
\newenvironment{example}
  {\pushQED{\qed}\examplex}
  {\popQED\endexamplex}

\theoremstyle{plain}
\numberwithin{equation}{section}

\newcommand{\CS}{{\C^*}}

\newcommand{\La}{\Lambda}

\newcommand{\In}{\subset}

\newcommand{\congto}{\xto{\sim}}
\renewcommand{\cong}{\simeq}

\newcommand*{\hrlen}{5}
\newcommand*{\hramp}{3}

\tikzset{
asdstyle/.style={blue,thick},
righthairs/.style={postaction={decorate,draw,decoration={border,amplitude=\hramp,segment length=\hrlen,angle=-90,pre=moveto,pre length=\hrlen/2}}},
lefthairs/.style={postaction={decorate,draw,decoration={border,amplitude=\hramp,segment length=\hrlen,angle=90,pre=moveto,pre length=\hrlen/2}}},
righthairsnogap/.style={postaction={decorate,draw,decoration={border,amplitude=\hramp,segment length=\hrlen,angle=-90}}},
lefthairsnogap/.style={postaction={decorate,draw,decoration={border,amplitude=\hramp,segment length=\hrlen,angle=90}}},
graphstyle/.style={thick},
arrowstyle/.style={thick,decorate,decoration={snake,amplitude=1.7,segment length=10pt,post length=.5mm,pre length=0}},
genmapstyle/.style={thick,-stealth'},
arrhdstyle/.style={thick},
exceptarcstyle/.style={red, ultra thick},
dualquiverstyle/.style={thick,->},
patstyle/.style={pattern color = gray, pattern = north east lines, opacity=0.3}
}

\begin{document}
\maketitle
\begin{abstract}
Coherent-Constructible Correspondence for toric variety assigns to each $n$-dimensional toric variety $X_\Sigma$  a Lagrangian skeleton $\La_\Sigma \In T^*T^n$,  such that the derived category of coherent sheaves $Coh(X_\Sigma)$ is equivalent to the (wrapped) constructible sheaves $Sh^w(T^n, \La_\Sigma)$. In this paper, we extend this correspondence, so that flip and flop between toric varieties corresponds to variation of Lagrangian skeletons. The main idea is to translate window subcategory in variation of GIT to a window skeleton.

    %Let $(\C^*)^k$ acts on $\C^N$ with weight $(a_1, \cdots, a_N)$, and let $\eta = \sum a_i$. It is well-known that there are two possible GIT quotients stacks $\xcal_\pm$, whose dg-derived categories $Coh(X_+)$ and $Coh(X_-)$ are related: if $\eta=0$, we have $Coh(X_+) \cong Coh(X_-)$; if $\eta>0$, we have $Coh(X_+) = \la T_\eta, \cdots, T_1, Coh(X_-)\ra$ and similary for $\eta <0$. It is also well-known that, using homological mirror symmetry, one can get mirror skeleton $\La_\pm \In T^* T^{N-1}$ such that $Sh^w(T^{N-1}, \La_\pm) \cong \Coh(\xcal_\pm)$. In this paper, we construct an interpolation family of skeleton $\{\La_t\}_{t \in [0,1]}$, such that if $\eta=0$, then variation $\La_t$ is non-characteristic ($Sh^w(T^{N-1}, \La_t)$ invariant) , and if $\eta \neq 0$ then there are $|\eta|$ many discontinuities in the family $Sh^w(T^{N-1}, \La_t)$. The construction is based on 'window subcategories' from variation of GIT quotients, and is closely related to the Floer theoretic approach by \cite{kerr2017homological}. 
\end{abstract}

\section*{Introduction}

\subsection{Variation of Skeletons}
Constructible sheaf has a nice balance between rigidity and flexibility. On one hand, to specify a constructible sheaf $F$ on a manifold $M$ with stratification  $M = \sqcup_\alpha S_\alpha$, one only needs to specify locally constant sheaf on each strata and glue them together. On the other hand, constructible sheaves on $M$ is equivalent to Lagrangians in $T^*M$ \cite{NZ, nadler2009microlocal, GPS3}. And all the flexibility of symplectic geometry, e.g. Hamiltonian isotopy, can be enjoyed by the constructible sheaves\cite{GKS2012sheaf}.

One manifestation of this flexibility is the existence of non-characteristic (NC) variation of skeleton, i.e, a family of Lagrangian skeleton $\La_t \in T^*M$ parameterized by $t$, such that the category of constructible sheaves with singular support in $\La_t$, denoted as $Sh^\dm(M, \La_t)$,  are invariant. %Here we also need to go beyond the notion of NC variation of skeleton and consider  families of categories can have "jumps". 

A natural source of NC variation of skeleton is   isotopy of Weinstein pair $(W, \hcal_t \in \pa W)$ \cite{GPS2, zhou2018sheaf}, where $W$ is a Weinstein domain (here $W = D^*M$ the unit disk bundle in $T^*M$), $\hcal_t$ is a family of Weinstein hypersurfaces in the contact boundary $\pa W$, and the family of skeletons arises as 
$$  \Skel(W, \hcal_t) = \Skel(W) \sqcup \\R_{>0} \cdot ( \Skel(\hcal_t)), $$
where $\R_{>0}$ acts by retracting Liouville flow. 
However, in practice it is often hard to write down the skeleton of a Weinstein pair and its variations. 

In this paper, we obtain a variation of skeleton not from symplectic geometry, but from a toric variation of GIT and coherent constructible correspondence (CCC). Here we only consider the simplest case of $\C^*$ acting on $\C^N$, and leave the general case for future work. There are two quotients toric smooth Deligne-Mumford stack $\xcal_\pm$, and correspondingly two skeletons $\La_\pm \In T^* \T^{N-1}$, where $\T = S^1$, with 
\be \label{eq:CCCpm} \Coh(\xcal_\pm) \cong Sh^w(\T^{N-1}, \La_\pm), \ee
Our goal is to find a variation of skeleton that interpolates $\La_\pm$, then get the semi-orthogonal decomposition relation between $Coh(\xcal_\pm)$. The answer comes naturally by translating the window subcategory in variation of GIT (VGIT) theory to Lagrangian skeleton.

\subsection{Setup and Results \label{ss:setup}}
Let $\C^*$ act on $\C^N$ with weight vector $a = (a_1, \cdots, a_N)$, that is 
\[ \C^* \times \C^N \to \C^N, \quad (t, z_1, \cdots, z_N) \mapsto (t^{a_1} z_1, \cdots, t^{a_N} z_N). \]
Denote $$ [N]=\{1,\cdots, N\}, \quad [N]_\pm= \{i \in [N]: \pm a_i > 0\}, $$ 
and let
$ \eta_\pm = \sum_{i \in [N]_\pm} |a_i|$. 
We assume from the beginning that
\[ a_i \neq 0, \quad gcd(\{|a_i|\})=1, \quad [N]_\pm \neq \emptyset, \quad \eta: = \eta_+ - \eta_- \geq 0. \]

There are two possible GIT quotient stacks:
$$ \xcal_\pm = [ (\C^N)_\pm^{ss}  \,/\, \C^*], \quad \z{ where } (\C^N)_\pm^{ss}  := \C^N \RM  \{ z \in \C^N \mid z_i= 0 \z{ if } i \in [N]_\pm \} $$
They are open substacks of $\xcal =  [\C^N \,/\, \C^*]$. We denote the inclusions as 
 $\iota_\pm: \xcal_\pm \into \xcal.$

\begin{example}
Let $\C^*$ acts on $\C^3$ with weight $(a_i) = (1, 1, -1)$. Then $(\C^3)^{ss}_+ = \{ (z_1, z_2, z_3) \in \C^3 \mid (z_1, z_2) \neq (0,0) \}$ and $(\C^3)^{ss}_- =  \{ (z_1, z_2, z_3) \in \C^3 \mid z_3 \neq 0 \}$. The resulting GIT quotients stacks are isomorphic to the more familiar toric varieties,  $\xcal_+ \cong \z{Tot}_{\P^1}(\ocal (-1))$ and $\xcal_- \cong \C^2$. 
\end{example}

From Coherent-Constructible-Correspondence (CCC) for toric variety (or DM stack) \cite{bondal2006derived, FLTZ-morelli, Ku16}, we have equivalences of dg categories \eqref{eq:CCCpm}.
The dg category of coherent sheaves on the two quotients are related by a semi-orthogonal decomposition \cite{kawamata-toric-3,BFK, halpern2015derived}, 
\be \Coh(\xcal_+) = \la \underbrace{\Vect, \cdots, \Vect}_{\eta \z{ times }}, \Coh(\xcal_-) \ra. \ee
Hence, we have corresponding decomposition from CCC,
\be Sh^w(\T^{N-1}, \La_+) = \la \underbrace{\Vect, \cdots, \Vect}_{\eta \z{ times }}, Sh^w(\T^{N-1}, \La_-) \ra. \ee
The question is to understand the semi-orthogonal decomposition purely in terms of constructible sheaves.

One way to obtain the semi-orthogonal decomposition on the B-side is through window subcategory of $\Coh (\xcal)$. 
Let $W = \{n, n+1, \cdots, n+k\}$ be a set of consecutive integers. We call $W$ a {\em window}. 
We define the {\bf B-model window subcategory} for $W$ to be the full triangulated subcategories 
$$ \bcal_W = \la \{ \lcal_i, i \in W\} \ra \subset \Coh(\xcal),  $$
where  $\lcal_k$ is the $\C^*$-equivariant line bundle over $\C^N$ with weight $k$. 

\bt[\cite{segal2011equivalences, BFK, halpern2015derived}]
Let $W$ be a window. Let $\iota_\pm: \xcal_\pm \into \xcal$ be inclusions of open stacks. 
\begin{enumerate}
    \item If $|W| = \eta_+$, then we have equivalence of categories
    $\iota_+^*|_{\bcal_W}: \bcal_W \congto \Coh(\xcal_+). $   
    \item If $ |W| = \eta_-$, then we have equivalence of categories
     $ \iota_-^*|_{\bcal_W}: \bcal_W \congto \Coh(\xcal_-).$ 
\end{enumerate}
\et

Also, if $W_1 \In W_2$, then we have fully faithful embedding $\bcal_{W_1} \into \bcal_{W_2}$.

If we fix $a \in \Z$, and consider windows
$$ W_+ = \{a,a+1, \cdots, a+\eta_+ - 1\}, \quad W_- = \{ a+\eta, \cdots, a+\eta_+ - 1\} \In W_+$$
then we have windows of sizes $\eta_+$ and $\eta_-$ respectively, and $W_-$ is aligned with $W_+$ on the right side. The inclusion of window subcategories $\bcal_{W_-} \into \bcal_{W_+}$ induces an embedding 
$$ (\iota_+^*) \circ (\iota_{-}^*)^{-1}:  \Coh(\xcal_-) \into \Coh(\xcal_+). $$ 
The $\eta$ copies of $\Vect$ corresponds to  structure sheaf of the unstable loci $(\C^N)_-^{us}$ in $\Coh([\C^N \qo \C^*])$ twisted by some character of $\C^*$. The details are reviewed in section \ref{ss:b-model}. 

Next we consider the $A$-side. 
From CCC, we have 
\begin{align*}
    \tau:  \Coh(\xcal) & \congto Sh^w(\ycal, \La_\ycal), \\
    \tau_\pm: \Coh(\xcal_\pm) & \congto Sh^w(\ycal, \La_{\ycal,\pm})
\end{align*} 
where $\ycal \cong \T^{N-1} \times \R$ and $\La_{\ycal,\pm} \In \La_\ycal \subset T^*\ycal$ are a conical Lagrangian of dimension $N$. The Lagrangian skeletons  $\La_\pm \subset T^* \T^{N-1}$ in \eqref{eq:CCCpm} arises as restriction of $\La_{\ycal,\pm}$ over $\T^{N-1} \times \{t\}$ for $t \in \Z$. 

Thus, it is natural to define the {\bf  A-model  window subcategory} as
$$ \acal_W = \tau(\bcal_W) = \la \{ \tau (\lcal_i), i \in W\} \ra \subset  Sh^w(\ycal, \La_\ycal). $$
The question is, how to characterize $\acal_W$ geometrically. 

We define the {\bf window (sub)skeleton $\La_W$} as the minimal sub-skeleton of $\La_\ycal$ that contains the CCC image of $\bcal_W$:
\be \La_W := \bigcup_{i \in W} SS(\tau(\lcal_i)) \subset \La_\ycal. \ee

Thus, we have natural fully faithful embedding of subcategories
$$ \iota_W:  \acal_W \into Sh^w(\ycal, \La_W). $$ 
It is not always true that the embedding is an equivalence. However, we have
\bt \label{thm:gen}
If $|W| \geq \eta_-$, then the fully faithful embedding $\iota_W$ is an equivalence of categories. 
\et

%Our next result is an analog of the symplectic reduction of $S^1$-action. Recall that if $S^1$ acts on $\C^N$ with weight vector $a$, then the moment polytope of the quotients can be obtained by taking slices of the polytope for $\C^N$. Namely, let $\pi_{a}: \R^N \to \R$ be the map of dot product with $a$, then $\pi_{a}^{-1}(1) \cap (\R_{\geq 0})^N$ is the moment polytope for $\xcal_+$ and $\pi_{a}^{-1}(-1) \cap (\R_{\geq 0})^N$ is the moment polytope for $\xcal_-$. 

Next, we consider the torus fibration 
$$ \mu: \ycal \to \R, \quad \ycal_t = \mu^{-1}(t) \cong \T^{N-1}. $$
Then we have a family of categories labelled by $t \in \R$
$$ \acal_{W}(t) = Sh^w(\ycal_t, \La_t), \quad \La_t: = \La_W|_{\ycal_t}.$$

The slices of $\La_W$ over $\ycal_t$ to the left and right of the window has nice relation with skeleton of GIT quotients. 
\bt[Theorem \ref{thm:slice}]
Let $W$ be a window. 
\begin{enumerate}
    \item If $|W| \geq \eta_-$, then over $\mu^{-1} (-\infty,  \min(-W))$, we have $\La_W = \La_{\ycal, -}$. 
    \item If $|W| \geq \eta_+$, then over $\mu^{-1} (\max(-W), +\infty)$, we have $\La_W = \La_{\ycal, +}$. 
\end{enumerate}
\et

\bt \label{thm:main}
Let $W$ be a window, and $|W| = \eta_+$. 

(a) If $\eta  = 0$, then $\acal_{W}(t)$ are constant (up to equivalence) for all $t \in \R$. 

(b) If $\eta > 0$, then there is a set $\jcal$ of size $\eta$, 
$$ \jcal = \{a, a+1, \cdots, a + \eta-1\}, \quad a = \min(-W),$$
such that $\acal_{W}(t)$ is locally constant (up to equivalence) away from $\jcal$. Moreover, 
\begin{enumerate}
    \item If $t > \max(\jcal)$, then $$\acal_{W}(t) \cong Sh^w(\T^{N-1}, \La_+).$$  
    \item If $t < \min(\jcal)$, then $$\acal_{W}(t) \cong Sh^w(\T^{N-1}, \La_-).$$
    \item If $t \in \jcal$, then for any small positive $\epsilon>0$, we have semi-orthogonal decomposition
    $$ \acal_{W}(t+\epsilon) = \la \Vect, \acal_{W}(t-\epsilon) \ra. $$
\end{enumerate}
\et

\subsection{Related works}

\sss{Variation of Lagrangian Skeleton}
Let $\La \In T^*M$ be a conical Lagrangian, and assume $T_M^*M \In \La$. Let $\La^\infty \In T^\infty M$ be the Legendrian in the contact infinity of $T^*M$. We speak of deformation of Legendrian $\La^\infty$ since $\La^\infty$ determines $\La$. 

The notion of non-characteristic variation of Lagrangian skeletons is introduced by Nadler \cite{nadler2015non}. In loc.cit, he proves any Legendrian singularity admits a NC deformation to a 'standard' (arboreal) singulairty, defined and studied in \cite{nadler2017arboreal}. 

One sufficient and necessary condition for NC variation of skeleton, is the that the complement of the Legendrian at infinity remains invariant upto contactomorphism \cite{GPS2, zhou2018sheaf}. Intuitively speaking, it is expected that, if the homotopy type of $\La^\infty$ is unchanged, and there is no birth or death of Reeb chords ending on $\La^\infty$ as it moves (for some fixed Reeb flow on $T^\infty M$), then the motion of $\La^\infty$ is non-characteristic. See also the recent work of Nadler-Shende on invariance of microlocal sheaf category under Weinstein isotopy \cite{nadler2020sheaf}. 

Deformation of Legendrians has been used to induce equivalences. Shende-Treumann-Williams-Zaslow have demonstrated beautifully how to understand cluster mutation from the point of view of Legendrian knot isotopy and the corresponding change of exact Lagrangian fillings \cite{STWZ-2019cluster-knot, STW2016combinatorics}. 

Given a Laurent polynomial with the coefficient in some tropical limit, one can construct a Lagrangian skeleton for the zero loci of the polynomial \cite{gammage2017mirror, zhou2020lagrangian}. If we rotate the arguments of the coefficients at the vertex of the Newton polytope, then the corresponding skeleton undergoes a NC deformation as well \cite{zhou2018sheaf, zhou2020lagrangian}. As well known, such rotation of the argument corresponds convolution on the constructible sheaf side\cite{FLTZ-morelli}, and also corresponds to tensoring by line bundle on the toric variety side (see also \cite{hanlon2019monodromy}). 

\sss{Homological Minimal Model Program} 
The homological mirror symmetry for toric varieties has been understood from different point of views, from Abouzaid's Lagrangian section \cite{abouzaid2009morse,abouzaid2006homogeneous}, through Fang-Liu-Treumann-Zaslow's treatment through constructible sheaves \cite{FLTZ-morelli, FLTZ-hms}, to  wrapped Fukaya categories  or wrapped microlocal sheaves \cite{Ku16}. 

There is an interesting refinement of homological mirror symmetry of toric variety, that not only matches the entire category of $Coh(X_\Sigma)$ to certain wrapped Fukaya category, but also matches certain semi-orthogonal decomposition of them. On the B-side (coherent side), it involves using a run  of minimal model program on $X_\Sigma$ to get a sequence of birational morphisms 
\[ X_\Sigma = X_0 \dashrightarrow X_1 \dashrightarrow \cdots \dashrightarrow X_n \]
as done by Kawamata \cite{kawamata-toric-3}, or Ballard-Favero-Katzarkov\cite{BFK}. On the A-side, suppose we are given a Landau-Ginzburg model $W: (\C^*)^n \to \C$, one may tropicalize $W$ (by making a choice), and decompose the critical values of $W$ into different sets, each lies approximately on a concentric circle \cite{DKK}.  Conjecturally, the choices of the minimal model run on the B-side matches the choices of the tropicalization of $W$ on the A-side, and  the semi-orthogonal decompositions of $\Coh(X_\Sigma)$ matches that with the $\Fuk( (\C^*)^n, W)$, and there has been progress towards it \cite{BFDKK, kerr2017homological}. 

Here we are proposing a third player in the picture, the constructible sheaf category. Given a minimal model run $X_0 \dashrightarrow \cdots \dashrightarrow X_n$, we can use coherent-constructible correspondence \cite{bondal2006derived, FLTZ-morelli, Ku16} to get a sequence of Lagrangian skeleta $\La_0, \cdots, \La_n$ in $T^*T^n$, however it is not clear (as far as the author know) how to build functor between $Sh(T^n, \La_{i})$ and $Sh(T^n, \La_{i+1})$ except in the special case that $\La_i \In \La_{i+1}$ or $X_i$ and $X_{i+1}$ is related by blow-up. The more general flip / flop between $X_i$ needs to be explained using variation of Lagrangian skeletons.  

Our final goal is to build functors between $Sh(T^n, \La_{i})$ and $Sh(T^n, \La_{i+1})$ using a continuous family of skeletons interpolating from $\La_i$ to $\La_{i+1}$. The general case will be studied in future work. 

In the upcoming paper with Jesse Huang \cite{HuangZhou2020}, we will consider the quasi-symmetric toric VGIT setup, a special case of toric CY, and one obtain an interesting local system of categories over a higher dimensional base.

\subsection{Sketch of the Proof}
Our main tool is the notion of variation of Lagrangian skeleton, which we now recall. 

Let $M$ be a real analytic compact manifold (we may relax the 'compact' condition in practice), $B$ a smooth manifold and $\pi: M \times B \to B$. 
A {\bf variation of Lagrangian skeleton} over $B$ is a Lagrangian $\La_B \In T^*(M \times B)$ containing the zero section, such that  \footnote{Here we use notation that $T^\infty M = (T^*M \RM  T_M^* M) / \R_+$ is the contact cosphere bundle at infinity, and $\La^\infty = (\La \RM T_M^*M) / \R_+$ is the (singular) Legendrian boundary of $\La$. }
$$ \La_B^\infty \cap T^\infty_{M \times \{b\}}(M \times B) = \emptyset, \quad \forall b \in B. 
$$
For $b \in B$, denote $M_b = \pi^{-1}(b)$ and $\La_b = \La_B|_{M_b}$. If $Sh^\dm(M_b, \La_b)$ is invariant as we vary $b$, then this is called 'non-characteristic' deformation by Nadler \cite{nadler2015non}. In general, we are interested in how $Sh^\dm(M_b, \La_b)$ varies. 

We study this problem locally in $M \times B$. We consider the sheaf of categories $Sh^\dm_{\La_B}$ on $M \times B$, and define its singular support $SS(Sh^\dm_{\La_B})$. We also define the push-forward of $\pi_*Sh^\dm_{\La_B}$, then
$$ SS(\pi_*Sh^\dm_{\La_B}) \In \pi_* (SS(Sh^\dm_{\La_B})) $$
where push-forward Lagrangian skeleton is by Lagrangian correspondence. 

In general, given a skeleton $\La \In T^*Z$, we are interested in how to compute $SS(Sh^\dm_\La)$ from $\La$. Let $V \In U \In Z$ be nested open sets, he restriction functor 
$$ \rho_{U, V}: Sh^\dm_\La(U) \to Sh^\dm_\La(V) $$
is an equivalence if and only if $\rho_{U,V}$ is fully-faithful and its left-adjoint $\rho_{U,V}^L$, i.e., the co-restriction functor is fully-faithful. Hence we defined two other singular supports $SS_{Hom}(\La)$ and $SS_{Hom}^L(\La)$ that measures the failure of fully-faithfulness. We have 
$$ SS(Sh^\dm_\La) = SS_{Hom}^L(\La) \cup SS_{Hom}(\La).$$

In our case, $M \times B = \T^{n-1} \times \R$, the variation of skeleton is $\La_W \In T^*(\T^{N-1} \times \R)$. Since the computation is local, we work in the universal cover $\wt \La_W \In T^*\R^N$. The computation of $SS_{Hom}(\La_W)$ boils down to an estimation of the singular support of the hom-sheaf. 
A useful trick in the estimation is the following theorem. 
\bt [Theorem \ref{thm:SSuhom} later]
Let $F, G$ be constructible sheaves over $M$ and $\{G_t\}_{t \in (0,\epsilon)}$ a positive isotopy of $G$. Then 
$$ SS(\uhom(F, G)) \subset \limsup_{t \to 0^+} SS(\uhom(F, G_t)). $$
In particular, if $SS^\infty(F) \cap SS^\infty(G_t) = \emptyset$ for $t \in (0, \epsilon)$ then 
$$ SS(\uhom(F, G)) \subset \limsup_{t \to 0^+} SS(F)^a + SS(G_t). $$
\et

The computation of $SS_{Hom}^L(\La_W)$, in our case, is done by a straight-forward, if tedious, study on the skeleton and the generators of the sheaf category locally, i.e, the microlocal skyscrapers sheaves. The main result is that, if the window size is large enough, more precisely $|W| \geq \eta_+$, then $SS_{Hom}^L(\La)$ is the zero-section. 

\subsection{Acknowledgement}
I would like to thank Gabe Kerr, Mathew Ballard for discussion about the VGIT problem; Yixuan Li, David Nadler and Vivek Shende for helpful discussions on the wrapped sheaf categories;  Jesse Huang for teaching me about perverse schober. The initial motivation to study variation of skeleton comes from discussion with Eric Zaslow on the work \cite{RSTZ}.  The construction of window skeleton is inspired by Colin Diemer and Maxim Kontsevich. This work is partially supported by IHES Simons Postdoctoral Fellowship as part of the Simons Collaboration on HMS.

\subsection{Notation\label{ss:notation}}

We summarize all the notations used in the paper for the ease of reference. 

1. $[N], [N]_\pm$ are defined in section \ref{ss:setup}. If $I \In [N]$, we let $I_\pm = I \cap [N]_\pm$. 

2. $\T = \R / \Z$. If $\NN$ is a lattice, we let $\NN_A := \NN \otimes_\Z A$ for abelian groups $A = \R, \C, \C^*, \T$. 

3. Let $\Z^N$ denote the character lattice of $(\C^*)^N$, with basis $e_i, i \in [N]$. Let $(\Z^N)^\vee$ denote  dual lattice,  with dual basis $e_i^\vee, i \in [N]$. 

4. For $I \In [N]$, we define closed cones
\[ \kappa_I = \cone(e_i, i \in I) \subset \R^N, \quad \sigma_I = \cone(e_i^\vee, i \in I) \subset (\R^N)^\vee. \]
We also define standard vectors 
\[ e_I = \sum_{i \in I} e_I, \quad e_I^\vee = \sum_{i \in I} e_I^\vee. \]

5. There are a dual pair of short exact sequences (SES) 
\[0 \to \Z \xto{\cdot a} (\Z^N)^\vee \to \NN \to 0  \] 
\[ 0 \to \MM \to \Z^N \xto{\mu_\Z} \Z) \to 0.\]
where $\mu_\Z(x) = \sum_i x_i a_i$. 
We tensor the second SES by $\R$ and quotient by $\MM$, to get
$$ 0 \to \frac{\MM_\R}{\MM} \to \frac{\R^N}{\MM} \xto{\mu}  \R \to 0. $$
We identify ${\MM_\R}\qo {\MM}$ with $\MM_\T$, then we get a torus fibration
\[ \mu: \ycal \to \R, \quad \ycal := \frac{\R^N}{\MM}, \quad \ycal_t := \mu^{-1}(t). \]

6. If $Z \subset \R^N$ is a locally closed subset, then let $\C_Z$ denote the constant sheaf with non-zero stalk $\C$ only on $Z$.

7. If $v \in \R^N$, we denote 
$$ Q_v = \C_{v + (\R_{>0})^N}$$  the constant sheaf supported on the shifted open quadrant. 

8. If $v \in \Z^N$, $I \subset [N]$, we denote
\[ \pcal_I(v) : = \uwave{Q_v} \to \bigoplus_{J \subset I, |J|=1} Q_{v - e_J} \to \bigoplus_{J \subset I, |J|=2} Q_{v - e_J} \to \cdots \to Q_{v - e_I}, \]
where $Q_v$ is sitting at degree $0$ \footnote{We always use wavy underline to indicate the degree $0$ component.} and the maps are induced by natural inclusion of open sets.

%If $I = [N]_\pm$, we use $\pcal_\pm(v)$. If $I = \emptyset$, we write $\pcal(v)$ instead of $\pcal_\emptyset(v)$.   The sheaf of complex $\pcal_I(v)$ will be co-representing the microlocal stalk functor at point $(v + \epsilon e_{I^c}, -e_I^\vee)$. 

9. We say $\C^*$ acts on $\C$ {\bf with weight $k$} if $\C^* \times \C \to \C$ is given by $(t,z) \mapsto t^k z$.
I
Let $\lcal$ be a $\C^*$-equivariant line bundle on $\C^N$. We say $\lcal$ is of weight $k$, if $\C^*$ acts on the fiber over fixed point $\lcal|_0$ with weight $k$. We denote the weight-$k$ line bundle by $\lcal_k$.  

10. If $\C^*$ acts on $\C$ with weight $1$, then $\C^*$ acts on its dual $\C^\vee$ with weight $-1$. Thus the space of function on $\C$ has $\C^*$-action with non-positive weights
$$ \ocal_\C = Sym^*(\C^\vee) = Sym^*(\C \cdot z) = \C[z], \quad \z{weight}(z)=-1. $$
If $\lcal_k$ is the weight $k$ line bundle on $\C$, then 
$$ \Hom(\lcal_0, \lcal_k) = \Gamma(\lcal_k) \cong \C[z] \cdot e_k, \quad \z{weight}(z)=-1, \z{weight}(e_k) = k. $$ 
We define the $\C^*$-equivariant map to be the weight-$0$ subspace above
$$ \Hom_\CS(\lcal_0, \lcal_k) = \Gamma(\lcal_k)_0 = \bcs \C \cdot z^k & k \geq 0 \\
0 & k < 0 \ecs. $$

11. If $Z$ is a real analytic manifold. We let $\pi_Z: T^*Z \to Z$ be its cotangent bundle, $\dot T^*Z = T^*Z \RM T^*_Z Z$ the punctured bundle, $T^\infty Z = T^*Z \qo \R_+$. If $\La \In T^*Z$ is a conical Lagrangian, we let $\La^\infty = (\La \cap \dot T^* Z) \qo \R_+$ be its Legendrian boundary. If $S \In Z$ is a submanifold, we let $T^\infty_S Z = (T^*_S Z)^\infty$.

If $S \In Z$ is an embedded submanifold, $\La \In T^*Z$ a conical Lagrangian, we define the {\bf restriction of $\La$ to $S$ as} 
$$ \La|_S  = q (\La \cap T^*Z|_S) \In T^*S $$
where $q$ is defined as 
$$ 0 \to T_{S}^* Z \to TZ|_{S} \xto{q} T^* S.  $$

Let $f: Z \to Y$ be a smooth submersion map, $Z_y = f^{-1}(y)$ for $y \in Y$,  $\La \In T^*Z$ a conical Lagrangian. We say $\La$ is {\bf non-characteristic with respect to $f$}, if 
$$ \La^\infty \cap T_{Z_y}^\infty = \emptyset. $$
We denote $\La_y = \La|_{Z_y} \In T^* Z_y$. 

\section{Examples} 

We present three examples of variation of Lagrangians. The first two are examples for non-characteristic deformation, i.e. $\eta=0$, for resolution of $\C^2/2$ and the Atiyah flop; the last one is $\C^*$ acting on $\C^2$ with weight $(3,-1)$, so that we can demonstrate the shape of the window skeleton and the behavior of the 'thimble'. The first two examples are also considered in the work of Donovan-Kuwagaki \cite{donovan-kuwagaki-2019mirror}, there the goal is to construct a perverse schober as 'a bridge' between the $\La_\pm$. %See also our discussion on perverse schober in Section \ref{ss:intro-schober}. 

\begin{example}
Consider $\C^*$ acts on $\C^3$ with weight $(1,1,-2)$. The two GIT quotient stacks are $\xcal_- = [\C^2/\Z_2]$ and $\xcal_+ = \ocal_{\P^1}(-2)$. 
The corresponding toric fans are shown below. 
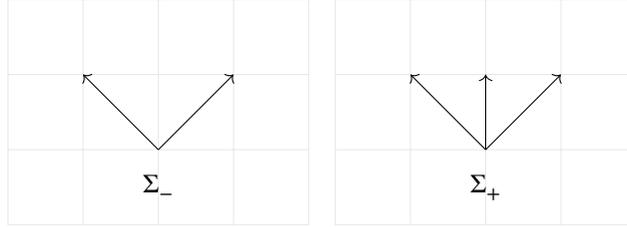
\[ 
\tikz{
 \draw[opacity=0.1 ,ultra thin] (-2,-1) grid (2,2);
 \draw [->] (0,0) -- (1, 1); 
 \draw [->] (0,0) -- (-1, 1); 
 \node at (0, -0.5) {$\Sigma_-$};
}
\quad
\tikz{
 \draw[opacity=0.1 ,ultra thin] (-2,-1) grid (2,2);
 \draw [->] (0,0) -- (1, 1); 
 \draw [->] (0,0) -- (0, 1); 
 \draw [->] (0,0) -- (-1, 1); 
 \node at (0, -0.5) {$\Sigma_+$};
}
\]
From this, we can build the skeletons $\La_-$ and $\La_+$ on $\T^2$. The following picture shows an open neighborhood of a fundamental domain of the unwrapping of $\T^2$. The hairs lines are the projection image of $\La^\infty \In T^\infty \T^2$, under the 'front projection' map $\pi: T^\infty \T^2 \to \T^2$. The small arcs indicate that the pre-image of $\pi|_{\La^\infty}$ over the corresponding intersection point is a small arc. 
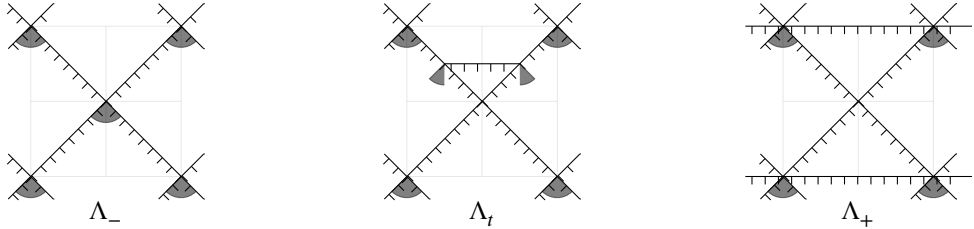
\[ 
\tikz{
\begin{scope}
 \draw[opacity=0.1,ultra thin] (0,0) grid (2,2);
 \draw [righthairs] (-0.3, -0.3) -- (2.3, 2.3); 
 \draw [righthairs] (-0.3,2.3) -- (2.3, -0.3);
\foreach\i/\j in {0/0,  2/2}{
\draw [righthairs] (\i -0.3,\j+0.3) -- +(0.6, -0.6);
\draw [fill, opacity=0.5] (\i, \j)-- +(-0.2, -0.2) arc (-135:-45:0.28);
}
\foreach\i/\j in {2/0, 0/2}{
\draw [righthairs] (\i -0.3,\j-0.3) -- +(0.6, 0.6);
\draw [fill, opacity=0.5] (\i, \j)-- +(-0.2, -0.2) arc (-135:-45:0.28);
}
\draw [fill, opacity=0.5] (1,1) -- (0.8, 0.8) arc (-135:-45:0.28);
 \node at (1, -0.5) {$\La_-$};
\end{scope};

\begin{scope}[shift={(5,0)}]
 \draw[opacity=0.1,ultra thin] (0,0) grid (2,2);
 \draw [righthairs] (-0.3, -0.3) -- (2.3, 2.3); 
 \draw [righthairs] (-0.3,2.3) -- (2.3, -0.3);
\foreach\i/\j in {0/0,  2/2}{
\draw [righthairs] (\i -0.3,\j+0.3) -- +(0.6, -0.6);
\draw [fill, opacity=0.5] (\i, \j)-- +(-0.2, -0.2) arc (-135:-45:0.28);
}
\foreach\i/\j in {2/0, 0/2}{
\draw [righthairs] (\i -0.3,\j-0.3) -- +(0.6, 0.6);
\draw [fill, opacity=0.5] (\i, \j)-- +(-0.2, -0.2) arc (-135:-45:0.28);
}
\draw [righthairs] (0.5,1.5) -- (1.5,1.5); 
\draw [fill, opacity=0.5] (0.5, 1.5)-- +(-0.2, -0.2) arc (-135:-90:0.28);
\draw [fill, opacity=0.5] (1.5, 1.5)-- +(0.2, -0.2) arc (-45:-90:0.28);
\node at (1, -0.5) {$\La_t$};
\end{scope}

\begin{scope}[shift={(10,0)}]
 \draw[opacity=0.1,ultra thin] (0,0) grid (2,2);
 \draw [righthairs] (-0.3, -0.3) -- (2.3, 2.3); 
 \draw [righthairs] (-0.3,2.3) -- (2.3, -0.3);
\foreach\i/\j in {0/0,  2/2}{
\draw [righthairs] (\i -0.3,\j+0.3) -- +(0.6, -0.6);
\draw [fill, opacity=0.5] (\i, \j)-- +(-0.2, -0.2) arc (-135:-45:0.28);
}
\foreach\i/\j in {2/0, 0/2}{
\draw [righthairs] (\i -0.3,\j-0.3) -- +(0.6, 0.6);
\draw [fill, opacity=0.5] (\i, \j)-- +(-0.2, -0.2) arc (-135:-45:0.28);
}
\draw [righthairs] (-0.5,0) -- (2.5,0); 
\draw [righthairs] (-0.5,2) -- (2.5,2); 
 \node at (1, -0.5) {$\La_+$};
\end{scope}
}
\]
An interpolation of skeletons, generated by the slices of window skeleton is of the form. 
\end{example}

\begin{example}
Atiyah Flop. Consider $\C^*$ acts on $\C^4$ with weight $(1,1,-1,-1)$. The toric fan $\Sigma_-$ and $\Sigma_+$ both have rays generated by 
$$ \{ v_1 = (0,0,1), v_2 = (0,1,0), v_3 = (0,0,1), v_4 = (1,1,-1)\}.$$
$\Sigma_-$ and $\Sigma_+$ differ at the triangulation of the (non-simplicial) $3$-dimensional cone generated by the four rays. In particular $\Sigma_-$ contains  $\cone(v_2,v_3)$, $\Sigma_+$ contains $\cone(v_1,v_4)$. 
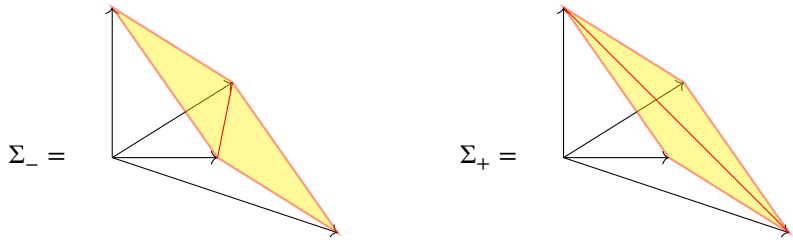
\begin{center}
   \begin{tikzpicture}
   \begin{scope}[x  = {(0.7cm, 0cm)},
                    y  = {( 0.8cm, 0.5cm)},
                    z  = {(0cm, 1cm)},scale=2 ]
    \draw [->] (0,0,0) -- (0,0,1); 
    \draw [->] (0,0,0) -- (0,1,0); 
    \draw [->] (0,0,0) -- (1,0,0); 
    \draw [->] (0,0,0) -- (1,1,-1); 
    \draw [draw=red, thick,   fill=yellow, opacity = 0.4] (0,0,1) -- (0,1,0) -- (1,1,-1) -- (1,0,0) -- cycle; 
     \draw [draw=red] (0,1,0) --  (1,0,0); 
    \end{scope}
       \begin{scope}[shift={(6,0)},x  = {(0.7cm, 0cm)},
                    y  = {( 0.8cm, 0.5cm)},
                    z  = {(0cm, 1cm)} ,scale=2]
    \draw [->] (0,0,0) -- (0,0,1); 
    \draw [->] (0,0,0) -- (0,1,0); 
    \draw [->] (0,0,0) -- (1,0,0); 
    \draw [->] (0,0,0) -- (1,1,-1); 
    \draw [draw=red, thick, fill=yellow, opacity = 0.4] (0,0,1) -- (0,1,0) -- (1,1,-1) -- (1,0,0) -- cycle; 
     \draw [draw=red] (0,0,1) --  (1,1,-1); 
    \end{scope}
    \node at (-1, 0) {$\Sigma_- = $};
    \node at (5, 0) {$\Sigma_+ = $};
    \end{tikzpicture}
\end{center}

The corresponding Lagrangians skeletons $\La_\pm$ and the interpolation is shown as following. The red thick line corresponds the image where $\pi|_{\La^\infty}$ is not finite. 
%Here $\T^3$ is shown as a unit cube, $x$-direction is to the right, $z$ upward and $y$ into the paper.  Each ray $v_i$ of the fan corresponds to a hypersurface $H_i$ (a subtorus) in the torus $\T^3$. The co-orientation of the hypersurface is opposite to that of the ray (recall the cones of the fan lives in the cotangent direction). The solid red lines marks the intersection of the hypersurfaces, such that over a smooth point $x$ of the red line, the Lagrangian is the cone of the corresponding conormal $v_i$ to the hypersurfaces that participate in the intersection. 

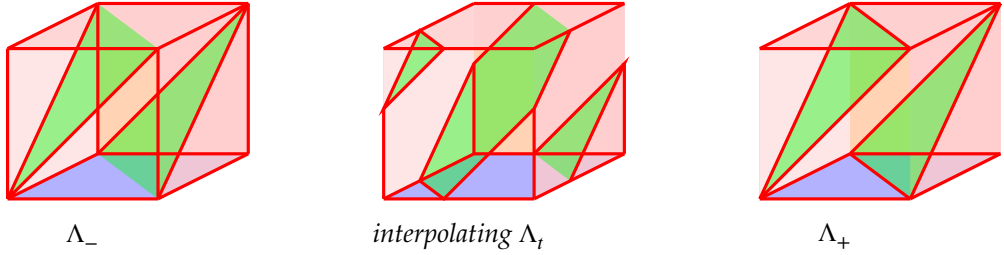
\newcommand{\Depth}{2}
\newcommand{\Height}{2}
\newcommand{\Width}{2}
\newcommand{\Dis}{0.8}
\begin{center}
\begin{tikzpicture}
\tikzset{linestyle/.style={red,very thick}}

\begin{scope} [x  = {(1cm, 0cm)},
                    y  = {( 0.6cm, 0.3cm)},
                    z  = {(0cm, 1cm)}
                    ]
\coordinate (O) at (0,0,0);
\coordinate (A) at (0,\Width,0);
\coordinate (B) at (0,\Width,\Height);
\coordinate (C) at (0,0,\Height);
\coordinate (D) at (\Depth,0,0);
\coordinate (E) at (\Depth,\Width,0);
\coordinate (F) at (\Depth,\Width,\Height);
\coordinate (G) at (\Depth,0,\Height);
\fill[ fill=yellow!80] (O) -- (C) -- (G) -- (D) -- cycle;% Bottom Face
\fill[ fill=blue!30] (O) -- (A) -- (E) -- (D) -- cycle;% Back Face
\fill[ fill=red!10] (O) -- (A) -- (B) -- (C) -- cycle;% Left Face
\fill[ fill=red!20,opacity=0.8] (D) -- (E) -- (F) -- (G) -- cycle;% Right Face
\fill[ fill=red!20,opacity=0.6] (C) -- (B) -- (F) -- (G) -- cycle;% Front Face
\fill[ fill=red!20,opacity=0.8] (A) -- (B) -- (F) -- (E) -- cycle;% Top Face
\fill[ fill=green, opacity=0.4]  (B) -- (O) -- (G) ; 
\fill[ fill=green, opacity=0.4 ](A) -- (F) -- (D);

\draw [linestyle] (O)--(C) (D)--(G) (E) -- (F) (A)--(B); 
\draw [linestyle] (O)--(A) (D)--(E)  (G) -- (F) (C)--(B); 
\draw [linestyle] (O)--(D) (A)--(E)  (B) -- (F) (C)--(G); 
\draw [linestyle] (O)--(B) (A)--(F)  (D) -- (F) (O)--(G); 
%% Following is for debugging purposes so you can see where the points are
%% These are last so that they show up on top
%\foreach \xy in {O, A, B, C, D, E, F, G}{
%    \node at (\xy) {\xy};
%}
\node at (1,0, -0.5) {$\La_-$};
\end{scope}

\begin{scope}[shift={(5,0)}, x  = {(1cm, 0cm)},
                    y  = {( 0.6cm, 0.3cm)},
                    z  = {(0cm, 1cm)}
                    ]
\coordinate (O) at (0,0,0);
\coordinate (A) at (0,\Width,0);
\coordinate (B) at (0,\Width,\Height);
\coordinate (C) at (0,0,\Height);
\coordinate (D) at (\Depth,0,0);
\coordinate (E) at (\Depth,\Width,0);
\coordinate (F) at (\Depth,\Width,\Height);
\coordinate (G) at (\Depth,0,\Height);

\fill[ fill=yellow!80] (O) -- (C) -- (G) -- (D) -- cycle;% Bottom Face
\fill[ fill=blue!30] (O) -- (A) -- (E) -- (D) -- cycle;% Back Face
\fill[ fill=red!10] (O) -- (A) -- (B) -- (C) -- cycle;% Left Face
\fill[ fill=red!20,opacity=0.8] (D) -- (E) -- (F) -- (G) -- cycle;% Right Face
\fill[ fill=red!20,opacity=0.6] (C) -- (B) -- (F) -- (G) -- cycle;% Front Face
\fill[ fill=red!20,opacity=0.8] (A) -- (B) -- (F) -- (E) -- cycle;% Top Face
\fill[ fill=green, opacity=0.4 ]   (0,0,2-\Dis) -- ++(\Dis, 0, \Dis) -- ++(-\Dis, \Dis, 0) -- cycle;
\fill[ fill=green, opacity=0.4]  ($(B) + (\Dis, 0,0)$) -- ($(B) + (0,0,-\Dis)$)  -- ($(O) + (0, \Dis, 0) $) --   ($(O) + (\Dis, 0, 0) $) --  ($(G) + (0,0,-\Dis)$) -- ($(G) + (0, \Dis, 0) $) -- cycle; 
\fill[ fill=green, opacity=0.4 ] ($(A)+ (\Dis, 0,0)$) -- ($(F) + (0,0,-\Dis)$)  -- ($(D)+ (0, \Dis, 0) $) ; 
%\fill[ fill=green, opacity=0.4 ] (0,2,0) -- (2,2,2) -- (2,0,0); 

%\draw [linestyle] (O)--(C) (D)--(G) (E) -- (F) (A)--(B); 
%\draw [linestyle] (B)--(G) (A)--(D); 
\draw [linestyle]  (0,0,2-\Dis) -- ++(\Dis, 0, \Dis) -- ++(-\Dis, \Dis, 0) -- cycle;
\draw [linestyle]   ($(B) + (\Dis, 0,0)$) -- ($(B) + (0,0,-\Dis)$)  -- ($(O) + (0, \Dis, 0) $) --   ($(O) + (\Dis, 0, 0) $) --  ($(G) + (0,0,-\Dis)$) -- ($(G) + (0, \Dis, 0) $); 
\draw [linestyle]  ($(A)+ (\Dis, 0,0)$) -- ($(F) + (0,0,-\Dis)$)  -- ($(D)+ (0, \Dis, 0) $) ;

\draw [linestyle] (O)--($(C) + (0,0,-\Dis)$) (D)-- ($(G) + (0,0,-\Dis)$) (E) -- ($(F) + (0,0,-\Dis)$) (A)--($(B) + (0,0,-\Dis)$); 
\draw [linestyle] (O)--(A) (D)--(E)  (G) -- (F) (C)--(B); 
\draw [linestyle] (O)--(D) (A)--(E)  (B) -- (F) (C)--(G); 
%\draw [linestyle] (O)--(B) (A)--(F)  (D) -- (F) (O)--(G); 
%% Following is for debugging purposes so you can see where the points are
%% These are last so that they show up on top
%\foreach \xy in {O, A, B, C, D, E, F, G}{
%    \node at (\xy) {\xy};
%}
\node at (1,0, -0.5) {interpolating $\La_t$};
\end{scope}

\begin{scope}[shift={(10,0)}, x  = {(1cm, 0cm)},
                    y  = {( 0.6cm, 0.3cm)},
                    z  = {(0cm, 1cm)}
                    ]
\coordinate (O) at (0,0,0);
\coordinate (A) at (0,\Width,0);
\coordinate (B) at (0,\Width,\Height);
\coordinate (C) at (0,0,\Height);
\coordinate (D) at (\Depth,0,0);
\coordinate (E) at (\Depth,\Width,0);
\coordinate (F) at (\Depth,\Width,\Height);
\coordinate (G) at (\Depth,0,\Height);
\fill[ fill=yellow!80] (O) -- (C) -- (G) -- (D) -- cycle;% Bottom Face
\fill[ fill=blue!30] (O) -- (A) -- (E) -- (D) -- cycle;% Back Face
\fill[ fill=red!10] (O) -- (A) -- (B) -- (C) -- cycle;% Left Face
\fill[ fill=red!20,opacity=0.8] (D) -- (E) -- (F) -- (G) -- cycle;% Right Face
\fill[ fill=red!20,opacity=0.6] (C) -- (B) -- (F) -- (G) -- cycle;% Front Face
\fill[ fill=red!20,opacity=0.8] (A) -- (B) -- (F) -- (E) -- cycle;% Top Face
\fill[ fill=green, opacity=0.4]  (B) -- (O) -- (G) ; 
\fill[ fill=green, opacity=0.4 ](A) -- (F) -- (D);

%\draw [linestyle] (O)--(C) (D)--(G) (E) -- (F) (A)--(B); 
\draw [linestyle] (B)--(G) (A)--(D); 
\draw [linestyle] (O)--(A) (D)--(E)  (G) -- (F) (C)--(B); 
\draw [linestyle] (O)--(D) (A)--(E)  (B) -- (F) (C)--(G); 
\draw [linestyle] (O)--(B) (A)--(F)  (D) -- (F) (O)--(G); 
%% Following is for debugging purposes so you can see where the points are
%% These are last so that they show up on top
%\foreach \xy in {O, A, B, C, D, E, F, G}{
%    \node at (\xy) {\xy};
%}
\node at (1,0, -0.5) { $\La_+$};
\end{scope}
\end{tikzpicture}
\end{center}
\end{example}

\begin{example}
Consider $\C^*$ acting on $\C^2$ with weight $(3,-1)$. The window skeleton is shown as below, living over $S^1 \times \R$ (drawn as $\R \times [0,1]$ with top and bottom edge identified). 

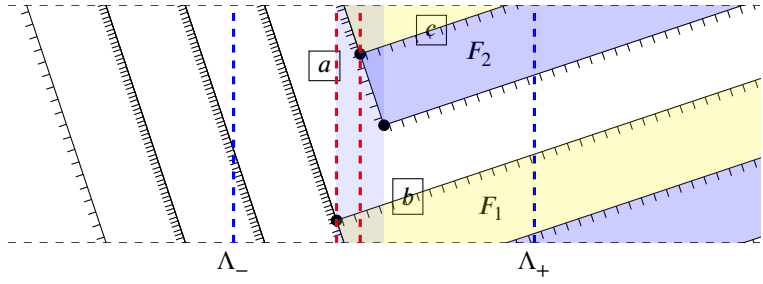
\begin{center}
\begin{tikzpicture}
\def\a{-1.264-0.3}
\def\b{1.897-0.3}

\begin{scope}
\clip (-5, \a) rectangle (5,\b);
\draw [dashed] (-5,\a) -- (5,\a); 
\draw [dashed] (-5,\b) -- (5,\b); 
\begin{scope} [rotate=18.5]
\foreach \c in {(-1,-1), (0,1),(0,0)} {
\foreach \v in {0,1,2,3}{  \begin{scope}[shift={(-\v, -\v*3)}] \draw [lefthairs,shift={\c}] (10,0) -- (0,0) -- (0, 20); \end{scope}}
\draw [fill, shift={\c}] (0,0) circle (2pt); 
}
\node[draw] at (-0.5, 1) {$a$};
\node[draw] at (0,-1) {$b$}; \node at (1,-1.5) {$F_1$};
\node[draw] at (1,1) {$c$};\node at (1.5,0.5) {$F_2$};

\fill [yellow, opacity=0.2] (-1,-2) rectangle (10,-1);
\fill [blue, opacity=0.2] (-1,-3) rectangle (10,-2);
\fill [yellow, opacity=0.2, shift={(1,3)}] (-1,-2) rectangle (10,-1);
\fill [blue, opacity=0.2, shift={(1,3)}] (-1,-3) rectangle (10,-2);

\end{scope}
\end{scope}

\def\cc{0.316}%cos(arctan(3))
\draw [red, very thick, dashed] (-\cc, \a) -- (-\cc, \b); 
\draw [red, very thick, dashed] (-2*\cc, \a) -- (-2*\cc, \b); 
\draw [blue, very thick, dashed] (-2, \a) -- (-2, \b); 
\draw [blue, very thick, dashed] (2, \a) -- (2, \b); 
\fill [blue, opacity=0.1] (-2*\cc,\a) rectangle (0, \b);
\node [below] at (-2,\a) {$\La_-$};
\node [below] at (2,\a) {$\La_+$};

\end{tikzpicture}\end{center}

The window skeleton is the union of three skeleton $\La(0), \La(1), \La(2)$, whose vertices are marked in black nodes. The window region is marked in shadow. Take a vertical slice on the right of the window region, we get the skeleton $\La_+$ for $[\C / \Z_3]$; and the vertical slice on the left of the window region gives skeleton $\La_-$ for $\C$.
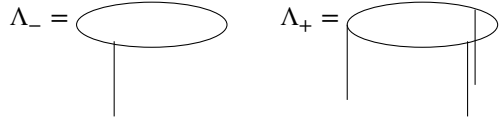
\[ \La_- =  \tikz[anchor=base, baseline]{
\draw (-0.5,-0.22) -- +(0,-1);
\draw (0,0) ellipse (1 and 0.3);
}\quad \quad 
\La_+ =  \tikz[anchor=base, baseline]{
\draw (-1,0) -- +(0,-1);
\draw (0.7,0.2) -- +(0,-1);
\draw (0.6,-0.22) -- +(0,-1);
\draw (0,0) ellipse (1 and 0.3);
} 
\]

As the slice pass through the window region, from  $t>2$ to $t<0$, the skeleton $\La_t$ under goes two jumps at $t=0$ and $t=1$, which are marked in red line. The change of skeleton is easy to understand, the three half rays in $\La_+$ gets merged to two rays, then to one ray. 

Next, we show how to identify the semi-orthogonal decomposition on the B-side
\[ \Coh([\C / \Z_3]) = \la \Vect_2, \Vect_1, \Coh(\C) \ra. \]
First, we claim that 
\[ Sh^w(\T \times \R, \La) \cong Sh^w(\T, \La_+) \cong \Coh([\C / \Z_3]) .\]
$\Coh(\C)$ corresponds to the full triangulated subcategory of $Sh^w(\T \times \R, \La)$ generated by sheaf $F_0$ co-representing the stalk at point $a$; $\Vect_1$ generated by the microlocal skyscraper sheaf
\footnote{To be precise, we should say 'a' microlocal skyscraper sheaf, since there are choices of grading. In practice, we fix an initial microstalk sheaf
\tikz[anchor=base, baseline, scale=0.7]{ 
\draw [lefthairs] (-1,0) -- (1,0); \draw [lefthairs] (-0.5, -0.3) -- (0.5, -0.3); \draw [fill, opacity = 0.1] (-0.5, -0.3) .. controls (0, 0.5) .. (0.5, -0.3);   \draw [lefthairs] (-0.5, -0.3) .. controls (0, 0.5) .. (0.5, -0.3);}, then apply geodesic flow stopped by the desired skeleton. 
} $F_1$ at point $b$ (support marked in yellow); $\Vect_2$ the microlocal skyscraper sheaf $F_2$ at point $c$ (support marked in blue). We can easily verify that $\Hom(F_i, F_j) = 0,  \text{ if } i < j$, and also the generation condition. Hence $F_2, F_1, F_0$ forms an exceptional collection and induces the semi-orthogonal condition. \end{example}

\section{Background}

\subsection{Wrapped Microlocal Sheaves}
We review the basic definitions and properties of wrapped microlocal sheaves, as introduced by Nadler in \cite{nadler2016wrapped} to mimic the definition of wrapped Fukaya categories. There is nothing original in this subsection, we only give an abridged account of loc.cit, section 1.1. 

We fix $k=\C$, or any other characteristic zero field. 

\sss{Motivation} First we introduce the notion of Lagrangian skeleton of a Weinstein manifold or Weinstein pair. Since these heuristics are only for motivations, we will not give the precise definitions but refer the reader to the survey \cite{eliashberg2017weinstein} and references in there. 

Let $(W, \omega = d\lambda)$ be an exact symplectic manifold, with Liouville one-form $\lambda$ and (retracting) Liouville vector field $X_\lambda$ defined by $\iota_{X_\lambda} \omega = - \lambda$. Roughly speaking, if the Liouville vector field 'points inward' near infinity and is gradient-like, then $W$ is a Weinstein manifold. The skeleton $\Skel(W)$ (or core) of a Weinstein manifold $W$ is the union of the unstable manifolds for $X_\lambda$. Hence  $\Skel(W)$ is a deformation retract of $W$ that is also compact and isotropic. We will only consider the case where the skeletons are Lagrangian. 

A variant of Weinstein manifold is {\bf Weinstein pair} $(W, H)$, where $H$ can be thought of as a Weinstein hypersurface in the contact boundary $\pa^\infty W$ of $W$ at infinity. Then $\Skel(W, H)$ is the union of $\Skel(W)$ and the flow-out of $\Skel(H)$ under the retracting flow $X_\lambda$. We assume $\Skel(W, H)$ is again a Lagrangian in $W$. 

One may define the {\bf (partially) wrapped Fukaya category} $\wcal(W, H)$ associated to the Weinstein pair $(W, H)$, where the objects are Lagrangian submanifolds, and morphisms $\Hom(L_1, L_2)$ are computed by wrapping $L_1$ positively near infinity.

Kontsevich outlined an approach to compute the category $\wcal(W, H)$ as global section of a cosheaf of categories on the skeleton $L = \Skel(W, H)$. This idea is realized first in the sheaf-theoretic
approach in \cite{nadler2016wrapped, shende2017microlocal} as cosheaf of dg categories $\mu sh_L^w$ of {\bf wrapped microlocal sheaf}, then as wrapped Fukaya category of Liouville sectors in \cite{GPS1, GPS2}.

\sss{Three categories of microlocal sheaves}
Let $\La \subset T^*Z$ be a closed conical Lagrangian in a cotangent bundle of a real analytic manifold $Z$. To a conical open subset $\Omega \In T^*Z$, we associate the following three dg categories. 
\begin{enumerate}
    \item $\mu Sh^\dm_\La(\Omega)$: the category of {\bf large} microlocal sheaves. It is a cocomplete dg category consists of microlocal sheaves with microstalks being arbitrary $k$-modules supported on $\La$. 
    \item $\mu Sh^w_\La(\Omega)$: the category of {\bf wrapped} microlocal sheaves . By definition, this is the category of compact objects in $\mu Sh^\dm_\La(\Omega)$. 
    $$ \mu Sh^w_\La(\Omega) = (\mu Sh^\dm_\La(\Omega))_c.$$
    \item $\mu Sh_\La(\Omega)$: the category of {\bf traditional} microlocal sheaves . This is a full dg subcategory of $\mu Sh^\dm_\La(\Omega)$, consisting of objects whose microstalks are perfect $k$-modules (i.e. finite dimensional $k$-vector spaces) supported along $\La$. 
\end{enumerate}

The wrapped microlocal sheaves $\mu Sh^w_\La(\Omega)$ can recover the other two.  The large microlocal sheaves is its ind-completion: $$\mu Sh^\dm_\La(\Omega) = Ind (\mu Sh^w_\La(\Omega)). $$
The traditional microlocal sheaves is the category of exact functors on it valued in perfect modules: 
$$ \mu Sh_\La(\Omega) = Fun^{ex}(\mu Sh^w_\La(\Omega)^{op}, \Perf_k). $$

There is a distinguished set of wrapped microlocal sheaves, called {\it microlocal skyscrapers}, that co-represents the microstalk functors associated to smooth strata of $\La \cap \Omega$. The collection of microlocal skyscrapers forms a compact generators of $\mu Sh^\dm_\La(\Omega)$, which split-generate $\mu Sh^w_\La(\Omega)$ \cite{nadler2016wrapped}[Lemma 1.4].

If $\Omega_1 \subset \Omega_2$ are two conic open subsets in $T^*Z$, we have the restriction functor
$$ \rho: \mu Sh^\dm_\La(\Omega_2) \to \mu Sh^\dm_\La(\Omega_1). $$
%The restriction functor preserves the traditional microlocal sheaves
%$$ \rho: \mu Sh_\La(\Omega_2) \to \mu Sh_\La(\Omega_1). $$
$\rho$ has a left-adjoint $\rho^L$ that preserves compact objects, we have then  $$\rho^L: \mu Sh^\dm_\La(\Omega_1) \to \mu Sh^\dm_\La(\Omega_2), \quad \rho^w: \mu Sh^w_\La(\Omega_1) \to \mu Sh^w_\La(\Omega_2).  $$ 
 
If $\La_1 \subset \La_2$ are two closed conic Lagrangians, then we have a fully-faithful embedding
$$ i: \mu Sh^\dm_{\La_1}(\Omega) \to \mu Sh^\dm_{\La_2}(\Omega). $$
$i$ also has a left-adjoint that preserves compact objects, we have
$$  i^L: \mu Sh^L_{\La_2}(\Omega) \to \mu Sh^L_{\La_1}(\Omega),  \quad  i^w: \mu Sh^w_{\La_2}(\Omega) \to \mu Sh^w_{\La_1}(\Omega). $$

The realization of microlocal sheaves are described in \cite{nadler2016wrapped}[Section 3]. Here we only use corresponding three versions for constructible sheaves, $Sh^\dm_\La$, $Sh^w_\La$ and $Sh_\La$. Given an open set $U \In Z$, we have $Sh^\dm_\La(U) = \mu Sh^\dm_\La(\pi^{-1}(U))$, where $\pi: T^*Z \to Z$ is the projection. All the statement about functors and left-adjoints applies here. 

\sss{Stalk Probe Sheaves}
If $U \In Z$ is an open set, and $p \in U$, then the stalk functor $\phi_p: Sh^\dm_\La(U) \to Mod_k$ is co-representable by an object $\fcal_p \in Sh^w_\La(U)$, i.e.
$ \Hom(\fcal_p, G) \cong G_p$ for all $ \gcal \in Sh^\dm_\La(U).$. We call $\fcal_p$ the {\bf  stalk probe sheaf} for point $p$ in $Sh^\dm_\La(U)$. This is an alias for Nadler's microlocal skyscraper sheaf in the special case that the microstalk functor is a stalk functor. 

In general, it is hard to explicitly describe the stalk probe sheaves. In this paper, we employ two methods to find the descriptions: 
\begin{enumerate}
    \item Geometric wrapping. We start from a sheaf $P_0 = \C_B$ supported on a small open ball around $p$, then we consider {\bf positive isotopy} $P_t$ of $P_0$, such that 
    $$ SS^\infty (P_t) \cap \La^\infty = \emptyset, \z{ and }, \; SS^\infty (P_1) \In \La^\infty. $$ 
    Or more generally, an arbitrary isotopy $P_t$ such that there exists a positive isotopy of $\La$, satisfying $SS^\infty (P_1) \In \La^\infty$ and the {\bf left non-characteristic (LNC)} condition
    $$ SS^\infty (P_t) \cap \La_s^\infty = \emptyset, \forall t \in [0,1], s \in (0,\epsilon). $$  
    \item Algebraic Constraint. If the stalk functor that we are interested in can be resolved using other stalks functors with known probe sheaves, then the probe sheaf that we are interested enjoys the same resolution. This is useful if we have know a closed embedding of skeleton $\La_1 \In \La_2$, and we know the probes for $\La_2$ and want to find probes for $\La_1$. 
\end{enumerate}

% For constructible sheaf, all the microstalk functors can be represented using chains of stalk functors. If a cotangent vector $p \notin \La$, then the chain complex of stalk functors representing the microstalk at $p$ is acyclic, and the corresponding chain of probe sheaves is acyclic. This can be used to express a probe using other known probes. 

% Another way to obtain the probe for a point $p \in U \In Z$ is to start from a small enough open ball $B$ around $p$, such that 
% $$ \Hom(\C_B, \gcal) = \phi_p(\gcal), \quad  $$ 
% Then one try to find a family of sheaves $\{\fcal_t\}_{t \in [0,1]}$ over $U$, such that $\fcal_0 = \C_B$, $\fcal_1 \in Sh^\dm_\La(U)$, and $\fcal_t$ satisfies the above condition as $\C_B$. This has been used for example in \cite{Zh17}.

\subsection{Coherent-Constructible Correspondence }
Let $\Sigma \In \R^n$ be a toric fan, $X_\Sigma$ be the corresponding toric variety (or smooth DM toric stack). The Coherent-Constructible Correspondence (CCC) \cite{FLTZ-morelli, Ku16} is  an equivalence of dg categories 
$$ \Coh(X_\Sigma) \cong Sh^w(\T^n, \La_\Sigma)$$
where $\La_\Sigma \In T^*\T^n$ is a Lagrangian skeleton that only depends on $\Sigma$. We say $\La_\Sigma$ is the {\bf mirror} to $X_\Sigma$. 

\bex
(1) The mirror of $\C^*$ is the zero-section $\La=\T=S^1 \In T^*\T$. The structure sheaf $\ocal_\CS$ is sent to $p_! \C_\R$, where $p: \R \to \T$ is the quotient map. 

(2) The mirror of $\C^*$ is the union of zero section with the half cotangent fiber over $0$,  $\{(x=0, \xi<0) \in T^* \T\}$. The structure sheaf $\ocal_\C$ is sent to (up to a shift in cohomological degree) $p_! \C_{(0,\infty)}$ (or equivalently $p_! \C_{(k, \infty)}$ for any $k \in \Z$), and the structure sheaf $\ocal_0$ for $0 \in \C$ is sent to $p_! \C_{(0,1]}$. Indeed
\[ \Hom_\T(p_! \C_{(0,1]}, p_! \C_{(0,1]}) = \Hom_\R( \C_{(0,1]}, p^! p_! \C_{(0,1]}) = \oplus_{k\in \Z}\Hom_\R( \C_{(0,1]}, \C_{(k,k+1]}) = \C \oplus \C[-1] \]
where the contribution for $\C[-1]$ is from summand $k=1$. 

(3) The mirror of $\P^1$ the union of zero section with the cotangent fiber $T_0^* \T$. The structure sheaf $\ocal_{\P^1}$ is sent to $\C_{\{0\}}$, $\ocal_{\P^1}(1)$ to $p_! \C_{ (0,1)}[1]$ and $\ocal_{\P^1}(-1)$ to $p! \C_{[0,1]}$. Here the choice of the cohomological degree is fixed by the requirement that tensor product on $Coh(\P^1)$ is sent to convolution product on $Sh(\T, \La)$. 
\eex

Now we establish some notations for toric geometry and describe the mirror skeleton $\La_\Sigma$ for a simplicial stacky fan $\Sigma$ following \cite{FLTZ-morelli}.

Let $\NN$ be a rank $n$ lattice, $\NN \cong \Z^n$ after we fix a basis of $\NN$. Let $\MM = \Hom_\Z(\NN, \Z)$ be the dual lattice. We use $\NN_\R$ for $\NN \otimes_\Z \R$, and similarly for $\NN$ replaced by $\MM$, $\R$ replaced by $\T, \C^*$. 

We recall the following definitions. 
\begin{enumerate}
\item  A convex polyhedral cone (abbreviated as 'cone' later) $\sigma \subset\NN_\R$ is a set of the form $\sigma = \cone(S) = \{\sum_{u \in S} \lambda_u u \mid \lambda_u \geq 0 \}$, where the cone generator $S \subset\NN_\R$ is a finite subset. A cone $\sigma$ is {\em rational} if there is a generator $S$ for $\sigma$ such that $S \subset N$. A cone is {\em strongly convex} if it does not contain any non-trivial linear subspace of $\NN_\R$.

\item A cone $\sigma \In \NN_\R$ is {\em simplicial} if $\sigma$ is generated by linearly independent vectors in $\NN_\R$. It is {\em smooth} if $\sigma$ is generated by a subset of $\Z$-basis of $\NN$. Let $\la -, - \ra$ denote the pairing $\MM \times \NN \to \Z$, and the induced other pairings, e.g $\MM_\T \times \NN \to \T$. 

\item Let $\sigma$ be a cone, we define the dual (closed) cone $\sigma^\vee$ as
\[ \sigma^\vee: = \{ x \in\MM_\R \mid \la x, y \ra \geq 0, \forall y \in \sigma \}. \]
We also define $\sigma^\perp = \{ x \in \MM \mid \la x, y \ra = 0, \forall y \in \sigma \} \subset\MM_\R$, and $\sigma^o$ (resp. $(\sigma^\vee)^o$) as the relative interior of $\sigma$ (resp. $\sigma^\vee$).
\item A face of a cone $\sigma$ is the subset $H_m \cap \sigma$ for some $m \in \sigma^\vee$ and $H_m = m^\perp$. We use $\sigma(r)$ to denote the collection of $r$-dimensional faces of $\sigma$. 
\item  A fan $\Sigma$ in $\NN_\R$ is a finite collection of strongly convex rational polyhedral cones $\sigma \subset\NN_\R$, such that (a) if $\sigma \in \Sigma$ then any face of $\sigma$ is in $\Sigma$, and (b) if $\sigma_1,\sigma_2$ are cones in $\Sigma$ then $\sigma_1 \cap \sigma_2$ is a face in both $\sigma_1$ and $\sigma_2$. We use $\Sigma(r)$ to denote the collection of $r$-dimensional cones in $\Sigma$. 
\item A fan $\Sigma$ in $\NN_\R$ is {\em complete}, if its support $|\Sigma| := \cup_{\sigma \in \Sigma} \sigma$ is the entire $\NN_\R$.  A fan $\Sigma$ is {\em simplicial (resp. smooth)}, if each maximal cone $\sigma \in \Sigma(n)$ is simplicial (resp. smooth). 
\end{enumerate}

\bd
A smooth stacky fan $(\Sigma, \{v_\rho\})$ is a simplicial fan $\Sigma$ together with a choice of non-zero lattice vector $v_\rho \in \rho$ for each ray $\rho \in \Sigma(1)$. We sometimes abbreviate $(\Sigma, \{v_\rho\})$ as $\Sigma$. 
\ed

Let $\Sigma$ be a smooth stacky fan. We recall the definition of the (non-equivariant) FLTZ skeleton
\[ \Lambda_{\Sigma} : = \bigcup_{\sigma \in \Sigma} \La_\sigma, \,  \z{ where } \quad \La_\sigma: = \MM_\T^\sigma \times (-\sigma) , \quad \MM_\T^\sigma: = \{ \theta \in \MM_\T \mid \la \theta, v_\rho \ra = 0,   \forall \rho \in \sigma(1) \}. \]
Thanks to $\sigma$ being simplicial, if $\sigma$ is $k$-dimensional, then $\MM_\T^\sigma$ is $(n-k)$-dimensional. %In general, $\MM_\T^\sigma$ has several connected components, labelled by the finite abelian group $(\NN \cap \spann_\R(\sigma))/(\spann_\Z \{ v_\rho, \rho \in \sigma(1)\})$.

Let $p: \MM_\R \to \MM_\T$ be the quotient map, which is a covering map with fiber $\MM$. Then $p$ induces a covering map on cotangent bundle, which we also denoted as $p: T^* \MM_\R \to T^*\MM_T$. Then we define the equivariant FLTZ skeleton $\wt \La_\Sigma$ as
\[ \wt \Lambda_{\Sigma} = p^{-1}(\Lambda_{\Sigma})
\]

The smooth DM toric stack for  $\Sigma$ is denoted by $\xcal_\Sigma$. Then non-equivariant and equivariant CCC says
\[ \tau: Coh(\xcal_\Sigma) \cong Sh^w(\MM_\T, \La_\Sigma), \quad \z{ and }  \tau_{eq}: Coh_{(\CS)^n}(\xcal_\Sigma) \cong Sh^w(\MM_\R, \wt \La_\Sigma). \]

\section{Variation of Lagrangian Skeletons}

Let $M$ be a real analytic manifold, $\La \subset T^*M$ a conical Lagrangian. We assume $\La$ always contains the zero-section $T^*_M M$, hence $\La$ is determined by its Legendrian boundary at infinity $\La^\infty \In T^\infty M$. 

In this section, we are going to consider a family of Lagrangians $\{\La_t\}_{t \in \R}$ and the corresponding wrapped (constructible) sheaves supported on it $Sh^w(M, \La_t)$. We are interested in how $Sh^w(M, \La_t)$ varies as $\La_t$ moves.  If $Sh^w(M, \La_t)$ remains invariant, we say the variation of  $\La_t$ is {\bf non-characteristic}, following Nadler \cite{nadler2015non}. For example, Nadler has shown that there exists non-characteristic deformation from arbitrary Legendrian singularities into certain particular nice ones called arboreal singularities. 

It is convenient to put the family of skeletons into one big skeleton $\La_\R \subset M \times \R$, such that the restriction of $\La_\R$ to the slice $M \times \{t\}$ equals $\La_t$,  and consider the categories $Sh^w_{\La_\R}(M \times J)$ for $J \subset \R$ open. This brings us to consider the following more general problem, that given $\La \subset T^* Z$ a skeleton, how will $Sh^w_\La(U)$, or equivalently $Sh^\dm_\La(U)$,  change as we vary the open subsets $U \In Z$.  

\subsection{Singular Support for Categories}

\bd[Singular Supports] \label{d:SScat}
Let $\La$ be a conical Lagrangian in $T^*M$. We define the {\bf singular support $SS(Sh^\dm_\La) \subset T^*M$ of the (co)sheaf of categories $Sh^\dm_\La$} as follows. A point $(x,\xi) \in T^*M$ is not in $SS(Sh^\dm_\La)$, if for any open neighborhood $U$ of $x$, there exists a smaller neighborhood $B$ of $x$ and smooth function $f: B \to \R$ such that $f(x)=0, df(x) = \xi$ and the natural restriction
\[ \rho_{f}: Sh_\La^\dm(\{f(x)<\epsilon\}) \to Sh_\La^\dm(\{f(x)< -\epsilon\}) \]
is an equivalence. We also abbreviate $SS(Sh^\dm_\La)$ to $SS(\La)$.
\ed

Recall that a functor $F: \ccal \to \dcal$ is an equivalence if and only if $F$ is fully-faithful and essentially surjective. And if $F$ admits a left adjoint, we have the following sufficient condition for essential surjectivity. 
\bl\label{lm:ess}
Let $F^L: \dcal \to \ccal$ denote the left adjoint of $F$. If $F^L$ is fully-faithful, then $F$ is essentially surjective. 
\el
\bpf
Let $X$ be an object in $\dcal$, by adjunction we have morphism $X \to F F^L X$ in $\dcal$. Since for any object $Y$ in $\dcal$, we have 
\[ \Hom(Y, X) \cong \Hom(F^L Y, F^L X) \cong \Hom(Y, F F^L X) \]
induced by the morphism $X \to F F^L X$, we see $X$ is quasi-isomorphic to $F F^L X$, hence $F$ is essentially surjective. 
\epf

\bd \label{d:SScat2}
As in the setup of Definition \ref{d:SScat}, we define $SS_{Hom}(\La)$ (resp. $SS_{Hom}^L$) in the same way as $SS(\La)$,  except we replace the condition that "$\rho_f$ is an equivalence" by "$\rho_f$ is fully-faithful" (resp. "$\rho_f^L$ is fully-faithful").
\ed

\bp \label{p:2cond}
\[SS(\La) = SS_{Hom}(\La) \cup SS_{Hom}^L(\La) \]
\ep
\bpf
Since if a functor is an equivalence, then its left-adjoint is also an equivalence. Thus if $(x,\xi) \notin SS(\La)$, then $(x,\xi) \notin SS_{Hom}(\La)$ and $(x,\xi) \notin SS_{Hom}^L(\La)$. Hence the inclusion $\supset$ is clear. 

On the other hand, if $(x,\xi) \notin SS_{Hom}(\La) \cup SS_{Hom}^L(\La)$, then by definition \ref{d:SScat2},  $\rho_f$ is a fully-faithful and $\rho_f^L$ is fully-faithful, hence $\rho_f$ is an equivalence of category by Lemma \ref{lm:ess}, thus $(x,\xi) \notin SS(\La)$.  Hence the inclusion $\subset$ is clear. 
\epf

There is another equivalent definition of $SS_{Hom}$. 
\bd
Let $\La_1, \La_2$ be two conical Lagrangian skeleton in $T^*M$, then
$$ \uhom(\La_1, \La_2) = \bigcup \{ SS(\uhom(F, G)) \mid  F \in Sh_{\La_1}^\dm(U), G \in Sh_{\La_2}^\dm(U), U \z{ open in } M \}.  $$
\ed

\bl\label{lm:sshom2}
\[ SS_{Hom}(\La) = \uhom(\La, \La) \]
\el
\bpf
 Recall $SS_{Hom}(\La)$ measures the failure of fully-faithfulness of the restriction functor, and $\uhom(\La, \La)$ measures the failure of the invariance of local section of hom sheaf between any two local objects, hence the two Lagrangian skeleton are the same. 
\epf

Let $M$ be a compact manifold, $B$ any manifold,  and $\pi: M \times B \to B$ be the projection, and $\La \In T^*(M\times B)$ be a conical Lagrangian. Then we may define a sheaf of categories on $B$ by pushing forward $Sh^\dm_{\La}$ on $M \times B$, 
$$ \pi_* Sh^\dm_{\La}:  U \mapsto Sh^\dm_{\La_B}(\pi^{-1}(U)). $$
Let $SS(\pi_* Sh^\dm_{\La})$ be defined in the same as in Definition \ref{d:SScat}. We have the following standard estimate
\[ SS(\pi_* Sh^\dm_{\La}) \subset \pi_* SS(Sh^\dm_{\La}) := \{ (b,\beta) \in T^*B \mid \exists (x, \xi; b, \beta) \in SS(Sh^\dm_{\La}), \xi=0\} \]

\subsection{Singular Support for Hom Sheaf}

First, we recall two results from \cite{KS} about upper bound of singular support of hom sheaf. 
\bl[\cite{KS}, Proposition 5.4.14] \label{lm:easybound}
Let $F, G$ are constructible sheaves on $M$. If $SS^\infty(F) \cap SS^\infty(G) = \emptyset$, then we have
\[ SS(\uhom(F, G)) \In SS(G) + SS(F)^a, \]
where $(-)^a: (x, \xi) \mapsto (x,-\xi)$ is the anti-podal map on $T^*M$ and $+$ is the fiberwise sum in $T^*M$. 
\el

\bl[\cite{KS}, Corollary 6.4.5, Remark 6.2.8] \label{lm:KShom2}
Let $F, G$ are constructible sheaves on $M$. Then we have
\[ SS(\uhom(F, G)) \In SS(G) \hat{+} SS(F)^a\]
where $(x_o, \xi_o) \in SS(G) \hat{+} SS(F)^a$ if and only if there exists sequences $\{(x_n, \xi_n)\}$ in $SS(G)$ and $\{(y_n, \eta_n)\}$ in  $SS(F)$, such that as $n \to \infty$, we have 
$ (i) x_n \to x_o, \quad y_n \to x_o$, 
$(ii) \xi_n - \eta_n \to \xi_o$, and 
$(iii) |x_n - y_n| |\xi_n| \to 0$
\el

Immediately, we get
\bc
\[ \uhom(\La, \La) \In \La \hat + \La^a \]
\ec
However, this bounded is not sharp. 

\bex \label{ex:cool}
Consider the following Lagrangian skeleton. We have  $\uhom(\La, \La) \subsetneq \La \hat + \La^a$. 
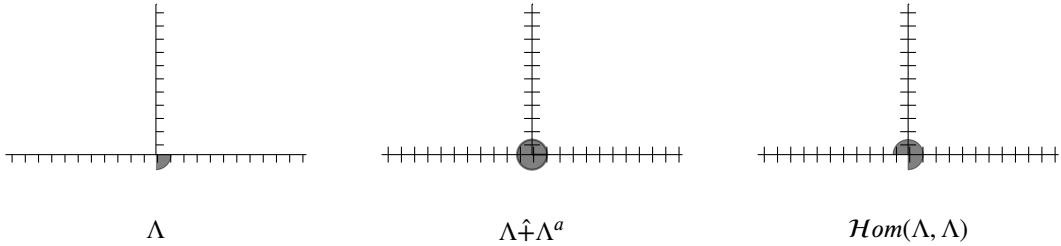
\begin{center}
\begin{tikzpicture}
\begin{scope}
\draw [righthairs] (-2,0) -- (2,0) (0,0) -- (0,2);
\draw [fill, opacity=0.5] (0,0) -- +(0:0.2) arc (0:-90:0.2);
\node at (0, -1) {$\La$}; 
\end{scope}

\begin{scope}[shift={(5,0)}]
\draw [righthairs, lefthairs] (-2,0) -- (2,0) (0,0) -- (0,2);
\draw [fill, opacity=0.5, thick] (0,0) -- +(0:0.2) arc (0:360:0.2);
\node at (0, -1) {$\La \hat + \La^a$}; 
\end{scope}

\begin{scope}[shift={(10,0)}]
\draw [righthairs, lefthairs] (-2,0) -- (2,0) (0,0) -- (0,2);
\draw [fill, opacity=0.5] (0,0) -- +(-90:0.2) arc (-90:180:0.2);
\node at (0, -1) {$\uhom(\La, \La)$}; 
\end{scope}

\end{tikzpicture}
\end{center}
where the shaded sector means the corresponding direction in the cotangent fiber $T_0^* \R^2$. 
\eex

The idea to improve is to perturb $G$ by a positive flow (e.g. a geodesic flow), so that $SS(G)^\infty \cap SS(F) = \emptyset$, then apply Lemma \ref{lm:easybound}.

\bex
Consider again the setup in Example \ref{ex:cool}. The translation of $\La^\infty$ in the direction $(1,-1)$ is a positive isotopy of Legendrian.  

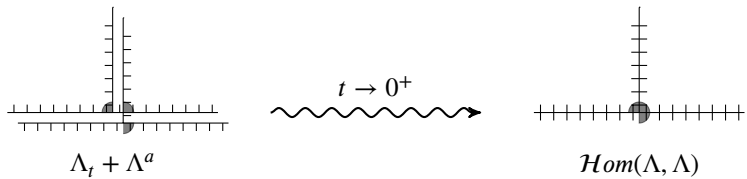
\begin{center}
\begin{tikzpicture} [scale=0.7]
\begin{scope}
\def\sf{0.2}
\draw [righthairs, shift={(\sf, -\sf)}] (-2,0) -- (2,0) (0,0) -- (0,2);
\draw [fill, opacity=0.5, shift={(\sf, -\sf)}] (0,0) -- +(0:0.2) arc (0:-90:0.2);
\draw [lefthairs] (-2,0) -- (2,0) (0,0) -- (0,2);
\draw [fill, opacity=0.5] (0,0) -- +(180:0.2) arc (180:90:0.2);
\draw [fill, opacity=0.5, shift={(\sf, 0)}] (0,0) -- +(0:0.2) arc (0:90:0.2);
\node at (0, -1) {$\La_{t} + \La^a$}; 
\end{scope}

\draw [arrowstyle, -stealth'] (3,0) -- (7,0); 
\node at (5,0.5) {$t \to 0^+$};

\begin{scope}[shift={(10,0)}]
\draw [righthairs, lefthairs] (-2,0) -- (2,0) (0,0) -- (0,2);
\draw [fill, opacity=0.5] (0,0) -- +(-90:0.2) arc (-90:180:0.2);
\node at (0, -1) {$\uhom(\La, \La)$}; 
\end{scope}

\end{tikzpicture}
\end{center}
\eex

Let $I = (0, \epsilon)$ for some $\epsilon>0$. Consider the open immersion and closed embedding
\[ j: M \times (0,\epsilon) \into M \times [0, \epsilon) \hookleftarrow M \times \{0\}: i. \]
We define the {\bf boundary value functor} to be
$$  \lim_{t \to 0^+}= i^{-1} j_*: \quad  Sh(M \times I) \to Sh(M), \quad F_I \mapsto F_0 = i^{-1} j_*(F_I). $$

 Consider space $M_1 \times M_2 \times I$, where $M_i = M$ and the subscript is only for bookkeeping sake, and $\pi_{M_i}$ or $\pi_{M_i \times I}$ the corresponding projection. 
\bd
Let $A_I \In M_1 \times M_2 \times I$ be a sub-analytic closed subset, $A_t := A_I \cap \pi_I^{-1}(t)$. Let $K = \C_{A_I}$ and $K_t = K|_{\pi_I^{-1}(t)}$. We say $K$ is a 
{\bf (upper-shriek) kernel for positive isotopy} if the following is true
\begin{enumerate}
\item For any $x \in M, t \in (0,\epsilon)$, $i = 1, 2$, $\pi_{M_i \times I}^{-1}(x,t)  \cap A_I$ is a  compact contractible set containing $x$. 
\item For $0 < t < s < \epsilon$, we have $A_t \subset A_s$. 
\item The intersection $A_0 := \bigcap_{0<t<\epsilon} A_t$ is  the diagonal subset $\Delta_M \In M \times M$.  
\item If $(x, \xi; y, \eta; t, \tau) \in SS(K)$ and is not contained in the zero section, then 
\[ \tau \geq 0, \quad \xi \neq 0, \quad \eta \neq 0. \]
\end{enumerate}
\ed

The kernel of positive isotopy acts as follows (see \cite{STW2016combinatorics} Appendix for a summary of the funtorial operation of sheaves)
\[ K^!: Sh(M_1) \to Sh(M_2 \times I), \quad K^! G : = (\pi_{M_2 \times I})_* \uhom(K, \pi_{M_1}^! G)\]
and 
\[ K_t^!: Sh(M_1) \to Sh(M_2), \quad K_t^! G : = (\pi_{M_2})_* \uhom(K_t, \pi_{M_1}^! G)\]
We have $$ G_t:=K_t^! G = (K^! G)|_{M_2 \times \{t\}}. $$ We call $G_t$ a {\bf positive isotopy of} $G$. And G is the boundary value of $G_I$ by construction
$$ G = \lim_{t \to 0^+} G_I. $$

For $0 < t < s < \epsilon$, we have canonical morphisms
\[ \rho_{t,s}: G_t \to G_s \]
called {\bf continuation morphisms} and is compatible with composition. Hence we have a projective system of sheaves. 

\bp \label{pp:projlim}
If $\{G_t\}_{0<t<\epsilon}$ is a positive isotopy of sheaves for $G=:G_0$, then the followings are true
\begin{enumerate}

\item The projective limit of $\{G_t\}_{0<t<\epsilon}$ in $Sh^\dm(M)$ exists and is $G$, 
\[ G = \varprojlim G_t,  \]

\item 
For any constructible sheaf $F \in Sh^\dm(M)$, we have projective limit of the projective system $\varprojlim \uhom(F, G_t)$, 
\[  \uhom(F, G) = \varprojlim \uhom(F, G_t) \]
\end{enumerate}
\ep

\bpf For (1), we observe that $K_0 = \C_{\Delta_M}$ is the inductive limit of $K_t$, $\uhom(K_t, \pi_{M_1}^! G)$ is the projective limit of $G$, and finally $\pi_{M_2 *}$ preserves projective limit.  (2) follows from (1), since hom and internal hom preserves projective limit. 
\epf

We can use the projective limit to get an upper bound of the singular support.
\bp\label{pp:SSproj}
If $G = \varprojlim G_t$ is the projective limit of a family of sheaves $G_t$, then we have
$$ SS(G) \subset \limsup_{t \to 0^+} SS(G_t) := \bigcap_{t_0 > 0} \overline{ \bigcup_{0 < t < t_0} SS(G_t) } 
$$
\ep
\bpf
Suppose $(x, \xi)$ is not in $\limsup_{t \to 0^+} SS(G_t)$, then there exists a conic open set $U$, and a $t_0>0$, such that for all $t < t_0$ $U$ is disjoint from $ SS(G_t)$. Hence we may choose an open ball $B$ around $x$, such that $B \subset \pi(U)$, and $f: B \to \R$, with $f(x)=0$, $df(x)=\xi$ and $\Gamma_{df} \subset U$. For small enough $\delta>0$, we consider the cone $P:= \C_{\{f(x)<-\delta\}} \to \C_{\{f(x)<\delta\}} $, then we have $\Hom(P, G_t) = 0$ for $t < t_0$. Since $\Hom(P, -)$ commutes with projective limit, we have $\Hom(P, G)=0$. This shows $(x,\xi)$ is not in $SS(G)$. 
\epf

\bt \label{thm:SSuhom}
Let $F, G$ be constructible sheaves over $M$ and $\{G_t\}_{t \in (0,\epsilon)}$ a positive isotopy of $G$. Then 
$$ SS(\uhom(F, G)) \subset \limsup_{t \to 0^+} SS(\uhom(F, G_t)). $$
In particular, if $SS^\infty(F) \cap SS^\infty(G_t) = \emptyset$ for $t \in (0, \epsilon)$ then 
$$ SS(\uhom(F, G)) \subset \limsup_{t \to 0^+} SS(F)^a + SS(G_t). $$
\et
\bpf
From Proposition \ref{pp:projlim}, we see $\uhom(F, G)$ is the projective limit of $\uhom(F, G_t)$. Then the conclusion follows from Proposition \ref{pp:SSproj}. The special case follows from Lemma \ref{lm:easybound}. 
\epf

\subsection{Left non-characteristic condition}
Here we introduce a technical condition that ensures $\Hom(F_t, G)$ to be invariant for an isotopy of sheaves $\{F_t\}$, even if $SS^\infty(F_t) \cap SS^\infty(G) \neq \emptyset$ for some $t$. 

First, we recall what is an isotopy of sheaves.  Let $I \In \R$ be an open interval, $M$ be a compact manifold, $F_I$ a constructible sheaf on $M \times I$. We say $F_I$ is an {\bf isotopy of sheaves} if $SS^\infty(F_I) \cap SS^\infty(\C_{M \times \{t\}}) = \emptyset$ for all $t \in I$. 

\bd
Let $F_I$ be an isotopy of sheaves, $G$ a constructible sheaf on $M$. We say $F_I$ is left non-characteristic (LNC) with respect to $G$, if there exists a positive isotopy of $G$, $\{G_s\}$, $s \in (0,\epsilon)$, such that
$$ SS^\infty(F_t) \cap SS^\infty(G_s) = \emptyset, \forall t \in I, s \in (0, \epsilon). $$
\ed

\bp
As in the above setup, if $F_I$ is LNC with respect to $G$, then 
 $\Hom(F_t, G)$ is invariant for $t \in I$.
\ep
\bpf
By non-characteristic deformation lemma (see e.g. \cite{}), we have $\Hom(F_t, G_s)$ to be invariant (up to quasi-isomorphism) for all $t \in I$ and $s \in (0,\epsilon)$. 

Then, we have $\Hom(F_t, G) = \varprojlim_s \Hom(F_t, G_s) = \Hom(F_t, G_{s_0})$ for some fixed $s_0$. Hence $\Hom(F_t, G)$ are invariant. 
\epf

\section{The Window Skeletons}
In this section, we will apply CCC and describe skeletons mirror to $\Coh(\xcal), \Coh(\xcal_\pm)$ and $\bcal_W$. They will be called {\bf full} skeleton, positive (or negative) {\bf cylindrical} skeleton and {\bf window} skeleton, respectively. 

%\noindent {\bf Notation:} Let $\Z^N$ denote the character lattice of $(\C^*)^N$,  and $(\Z^N)^\vee$ the lattice for 1-parameter subgroup (1-PS).  Let $\R^N = \Z^N \ot_\Z \R$, $(\R^N)^\vee = (\Z^N)^\vee \ot_\Z \R$. 
%We use $e_1, \cdots, e_N$ as basis standard of $\R^N$, and $e_i^\vee$ as dual basis on $(\R^N)^\vee$. Let $[N]=\{1, \cdots, N\}$. 
%If $I \subset [N]$, we let $e_I = \sum_{i \in I} e_i$, $e_I^\vee = \sum_{i \in I} e^\vee_i$. We use $\bfone = e_{[N]}$.  We define cones $\kappa_I = \cone(e_i, i \in I) \subset \R^N$ and $ \sigma_I = \cone(e_i^\vee, i \in I) \subset (\R^N)^\vee$. We denote $\kappa_\pm = \kappa_{[N]_\pm}$ and $\sigma_\pm = \sigma_{[N]_\pm}$. We let $V_I$ be the $|I|$ dimensional complex subspace, parametrized by $\{z_i, i \in I\}$, and $V_\pm = V_{[N]_\pm}$. 

\subsection{The full skeleton and the GIT quotient skeleton}
First, we recall the  fan $\Sigma_N$ for $\C^N$.  $\Sigma_N$ consist of all the faces of the closed positive quadrant $\R_{\geq 0}^N$
\[ \Sigma_N = \{ \sigma_I \mid I \subset [N] \}, \quad \sigma_I = \cone(e_i^\vee, i \in I ) \}.  \]

Consider the 1-PS $\CS \into (\CS)^N$ given by $a \in (\Z^N)^\vee$. It induces $\Z^\vee \into (\Z^n)^\vee$ and  $\mu_\Z: \Z^N \onto \Z$, and we define $\MM$ and $\NN$ by
\[ 0 \to \Z \xto{\cdot a} (\Z^N)^\vee \xto{q_\Z} \NN \to 0,  \] 
\[ 0 \to \MM \to \Z^N  \xto{\mu_\Z} \Z \to 0.  \]
Since $a$ is a primitive vector, $\NN$ and $\MM$ are free abelian group. We denote $\NN_A = \NN \ot_\Z A$ and similarly for $\MM_A$, for abelian group $A = \R, \C^*,\T = \R/\Z$. 

We define sub-fans $\Sigma_N^\pm \subset \Sigma_N$: 
$$ \Sigma_N^\pm := \{ \sigma_I \mid I \In [N], \z{s. t. } [N]_\pm \not \In I \}. $$
Consider the quotient map 
$$ q_\R: (\R^N)^\vee \to \NN_\R $$
Define the image of $(\Sigma_N^\pm, \{e_i^\vee: \sigma_{\{i\}} \in \Sigma_N^\pm\})$ under the quotient by $(\Sigma_\pm, \{v_i\})$. 
This is a smooth stacky fan. 

Let $\xcal_\pm$ denote the $(N-1)$-dimensional smooth toric DM stack for $\Sigma_\pm$. And let $\wt \xcal_\pm = (\C^N)^{ss}_\pm \subset \C^N$ denote the smooth open subvariety for $\Sigma_N^\pm$.

Consider the torus fibration
$$ \MM_\T \to \R^N / \MM \xto{\mu} \R. $$
Denote 
$$ \ycal = \R^N / \MM, \quad \ycal_t = \mu^{-1}(t). $$
For $t \in \Z$, $\ycal_t \cong \MM_\T$ canonically. 

The $(\C^*)^N$-equivariant skeleton for $\C^N$ is $\La_N \subset T^* \R^N$
\[ \La_N = \bigcup_{\sigma_I \in \Sigma_N} (\Z^N + \sigma_I^\perp) \times (-\sigma_I) \In T^*\R^N. \]
The $\C^*$-equivariant skeleton for $\C^N$ is the skeleton for $\xcal =[\C^N / \C^* ]$, 
\[ \La_N / \MM = \bigcup_{\sigma_I \in \Sigma_N} [{(\Z^N + \sigma_I^\perp)}/{\MM}] \times (-\sigma_I) \subset T^*Y. \]
The $\C^*$-equivariant skeleton for $\wt \xcal_\pm$ is the skeleton for $\xcal_\pm =[\wt \xcal_\pm / \C^* ]$, 
\[ \La_N^\pm / \MM = \bigcup_{\sigma_I \in \Sigma_N^\pm} [{(\Z^N + \sigma_I^\perp)}/{\MM}] \times (-\sigma_I) \subset T^*Y. \]
Equivalently, we can consider the non-equivariant skeleton for the stacky fan $\Sigma_\pm$
\[ \La_\pm := \La_{\Sigma_\pm} = \bigcup_{\sigma \in \Sigma_\pm} \MM_\T^\sigma \times (-\sigma) \In T^*\MM_\T. \]
Then $\La_\pm$ is the reduction of $\La_N^\pm / \MM$ on $\ycal_t$ for $t \in \Z$. 

We call $\La_N \qo \MM$ the {\bf full skeleton}, $\La_N^\pm \qo \MM$
the {\bf cylindrical GIT quotient skeleton} and $\La_\pm$ simply the {\bf GIT quotient skeleton}.

\ss{The Window Skeleton: Periodic region}
Let $W$ be a window, i.e., a set of consecutive integers. We define the window skeleton as a union of singular supports of mirror to the objects in the B-model mirror. 

Under equivariant CCC, a $(\C^*)^N$-equivariant line bundle $\lcal_k$ on $\C^N$, i.e $(\C^*)^N$ acts on the fiber over $0$ by weight $k$, goes to a constructible sheaf $Q_{-k}$, where
$$ Q_v = \C_{v + (\R_{>0})^N}$$
the constant sheaf supported on the shifted open quadrant. 

Hence, we define a subset of $\Z^N$ by
$$ \Z^N(W) := \mu_\Z^{-1}(-W). $$
Then we define the window skeleton upstairs
$$ \wt \La_W = \bigcup_{v \in \Z^N(W)} SS(Q_v), \quad \La_W = \wt \La_W \qo \MM.$$

\bt\label{thm:slice}
Let $W$ be a window and $-W  = \{a, a+1, \cdots, b\}$. 
\begin{enumerate}
    \item If $|W| \geq \eta_+$, then  
    $$ \wt \La_W|_{(b, \infty)} = \La_N^+|_{(b, \infty)} $$
    \item If $|W| \geq \eta_-$, then $$ \wt \La_W|_{(-\infty, a)} = \La_N^+|_{(-\infty, a)}  $$
\end{enumerate}
\et
\bpf
We only prove (1) since the proof for (2) is similar. To show $\wt \La_W|_{(b, \infty)} \subset \La_N^+|_{(b, \infty)}$, we only need to show that $SS(Q_v) |_{(b, \infty)} \subset \La_N^+|_{(b, \infty)}$ for all $v \in \Z^N(W)$. We have decomposition
$$ SS(Q_v) = \bigcup_{I \subset [N]} (v + \kappa_{I^c}) \times (-\sigma_I) $$
If $[N]_+ \subset I$, then $I^c \subset [N]_-$, hence $\mu(v + \kappa_{I^c}) \leq mu(v) \leq b$, thus 
$$ SS(Q_v) |_{(b, \infty)} |_{(b, \infty)}  =  \cup_{I \subset [N], [N]_+ \not \In I} (v + \kappa_{I^c}) \times (-\sigma_I).$$
Hence $SS(Q_v) |_{(b, \infty)} \subset \La_N^+|_{(b, \infty)}$.

Next, we show $\wt \La_W|_{(b, \infty)} \supset \La_N^+|_{(b, \infty)}$. We only need to show that, for any $u \in \Z^N$, for any $I \subset [N]$ with $[N]_+ \not \In I$, 
$$ \wt \La_W|_{(b, \infty)} \supset (u + \sigma_I^\perp)|_{(b, \infty)} \times (-\sigma_I). $$ 
Again, we only need to show that for any $x \in (u + \sigma_I^\perp)|_{(b, \infty)}$, there exists a $v \in \Z^N(W)$, such that
\be \label{eq:vxi} SS(Q_v) \supset x \times (-\sigma_I). \ee
Since $x - u \in \sigma_I^\perp$, we have
$$ x_i \in \Z \quad \z{ if} i \in I. $$
We may write $x_i = [x_i] + \{x_i\}$, where $[x_i] \in \Z$ and $\{x_i\} \in [0, 1)$. Then
$$ \mu([x]) = \mu(x) - \mu(\{x\}) > b - \sum_{j} a_j \{x_j\} =  b - \sum_{j \in I^c} a_j \{x_j\} > b - \sum_{j \in I^c \cap [N]_+} a_j \{x_j\} > b - \sum_{j \in [N]_+} a_j = b- \eta_+ $$
Since $\mu([x])$ is an integer, we have $\mu([x]) \geq b- \eta_+ +1$. 
If $\mu([x]) \leq b$, then we may choose $v = [x]$; otherwise, we may choose $v = [x] - k (\sum_{j \in I^c \cap [N]_+} e_j)$ for certain positive integer $k$ such that $\mu(v) \in [a,b]$. Indeed this is possible, since $I^c \cap [N]_+ \neq \emptyset$, and $\eta_+ \geq \mu(\sum_{j \in I^c \cap [N]_+} e_j) > 0$, hence $\mu([x] - k (\sum_{j \in I^c \cap [N]_+} e_j))$ decrease with step size at most $\eta_+$, hence will eventually fall into the interval $[a,b]$ which contains at least $\eta_+$ integers. 

Now that we have a $v \in x + (\sigma_I)^\perp$ and $v_i \leq x_i$ for all $i$, we have Eq \ref{eq:vxi} as desired.

%We claim that $u -  \kappa_{I^c} \cap \Z^N(W) \neq \emptyset$. Indeed, $\kappa_{I^c}$ is a cone  generated by $e_i$ with $i \in I^c$. Since $[N]_+ \not \In I$, we have $I^c \not \In [N]_-$, hence there is at least  

% %If $\sigma \in \La_{N}^+$ and ${\bf u + x} \in (\Z^N + \sigma^\perp) \cap \mu^{-1}((b,\infty))$, then without loss of generality, we may assume $x_i \in [0,1)$. Suppose $\sigma = \cone(e_j, j \in J)$ for some $J \subset \{1, \cdots, N\}$, then there exists some $j \notin J$ with $a_j>0$, and $x_i=0$ for all $i \in J$. Hence
% \[ \mu( {\bf u}) = \mu( {\bf u+x}) - \mu({\bf x}) > b - \sum_{i \notin J} a_i \geq b - \sum_{i \notin J, a_i > 0} a_i >  b - \sum_{a_i > 0} a_i = b - \eta_+. \]
% Thus, there exists some non-negative integer $k$, such that 
% ${\bf v} = {\bf u} - k(\sum_{j \notin J, a_j>0} e_j) $
% satisfies $b-\eta_+ < \mu(\bf v) \leq b$. We may check that $SS(\C_{{\bf v}+(\R_{>0}^N)}) \supset ({\bf u + x}) \times (-\sigma)$, hence the other inclusion is verified. 
\epf

\subsection{Window Skeleton near a Lattice Point \label{ss:winskel}}
Let $W$ be a window with $-W=\{a,\cdots, b\}$. We now study the window skeleton $\wt \La_W|_{(a-\epsilon,b+\epsilon)}$ for some small $\epsilon$. This is determined by the behavior near a lattice point $v \in \Z^N(W)$. 

In general, for any $v \in \Z^N$,  we may consider the specialization of $\wt \La_W$ at $v$
$$ \wt \La_{W,v} \subset T^*( T_v \R^N) = \R^N \times (\R^N)^\vee. $$
Since the strata of $\wt \La_W$ is already conical near $v$, $\wt \La_{W,v}$ can be obtained by first restricting to a small ball centered at $v$ then extend back to $\R^N$.

We need some notation. For $I \subset [N]$, we define 
$$ P_I: = \{x \in \R^N: x_i > 0, \z{ if }  i \notin I\}, \quad \La_I = SS(\C_{P_I}). $$
We call  $I$ positive-type (resp. negative-type) if $I \subset I_+$ (resp. $I \subset I_-$), and we call $I$ mixed-type, if $I$ intersects with both $[N]_\pm$ nontrivially. 

If $\La_I \subset \wt \La_{W, v}$, we say {\bf $\La_I$ appears in $\wt \La_{W,v}$}. We denote the set of $I$ where $\La_I$ appears in $\wt \La_{W,v}$ as $\ical_{W,v}$. 

\bp \label{pp:La_I}
$\La_I$ appears in $\wt \La_{W,v}$ if and only if there exists a vertex $w \in \Z^N(W)$, such that 
\[ 
\bcs
w_i < v_i & i \in I \\
w_i = v_i & i \notin I
\ecs
\] 
\ep

\bpf
Since $\wt \La_W$ is a union of $\La_{w}$ for each $w \in \Z^N(W)$. And only if $w \leq v$ can $\La_{w}$ contribute to $\La_W|_v$. For $\La_{w}$ contribute as $\La_I$, we need the quadrant $Q_w$ to equal to $Q_I$ near $v$, hence the conditions on $w_i$. 
\epf

\bl\label{lm:mix}
Let $v \in \Z^N(W)$ and $I \subset [N]$. If $I$ is of mixed type,  or if $I = \emptyset$, then $\La_I$ appears in $\wt \La_{W,v}$. 
\el
\bpf
By previous proposition, it suffices to choose $w$ with stated property. If $I=\emptyset$, then we may choose $w=v$. If $I$ is of mixed type, then it suffices to find positive integers $c_i$ for all $i \in I$, such that 
$ \sum_{i \in I} a_i c_i = 0.$ It is easy to see such solution exists. Then, we may choose $w = v - \sum_{i \in I} c_i {\bf e_i}$. 
\epf

\bl \label{lm:mix2}
Let $v \in \Z^N(W)$ and $I \subset [N]$. If $I$ is of positive-type or negative-type, then $\La_I$ appears in $\wt \La_{W,v}$ if and only if the vertex $w = v - \sum_{i \in I} a_i$ is in $\Z^N(W)$. 
\el
\bpf
From Proposition \ref{pp:La_I}, we may verify the 'if' part of the statement using the $w$ given. For the 'only if' part, suppose $\La_I$ appears in $\wt \La_{W,v}$, then there is $w' = v - \sum_{i \in I} a_i c_i$ for positive integers $c_i$, and $\mu(w') \in -W$.   If $I$ is of positive-type, we have $\mu(v) > \mu(w) \geq \mu(w')$; and if $I$ is of negative-type, we have $\mu(v) < \mu(w) \leq \mu(w')$, in either case we have $\mu(w)$ sandwiched between two integers in a set of consecutive integers $-W$ , hence $\mu(w) \in -W$. 
\epf

\bl \label{lm:leftright}
Let $v \in \Z^N$ and $I \subset [N]$.
\begin{enumerate}
    \item If $|W| \geq \eta_+$ and $\mu(v) > \max(-W)$, then $\La_I$ appears in $\wt \La_{W,v}$ if and only if $I$ is of mixed-type or $I$ is of positive-type.  
    \item If $|W| \geq \eta_-$ and $\mu(v) < \min(-W)$, then $\La_I$ appears in $\wt \La_{W,v}$ if and only if $I$ is of mixed-type or $I$ is of negative-type. 
\end{enumerate}
\el
\bpf
We only prove the statement (1), the other one is similar. 

If $I$ is of mixed-type or positive-type, then we just need to find a $w \in \Z^N(W)$ satisfies the conditions in Proposition \ref{pp:La_I}. Suppose $w = v - \sum_{i \in I} c_i e_i$, for positive integers $c_i$, then 
$ \mu(w) = \mu(v) - \sum_{i \in I} c_i a_i$. We need to show that there exists $c_i$s such that $\mu(w) \in -W$. Let 
$$  \alpha_I=  \sum_{i \in I} a_i, \quad \beta_I = \min ( \Z_{>0} \cap \{ \sum_{i \in I} a_i c_i \mid c_i \geq 0 \} ).$$
Then $\beta_I \leq \eta_+$ and $\alpha_I \leq \eta_+$, and
$$ \{ \mu(v) - \sum_{i \in I} c_i a_i \mid c_i \in \Z_{>0} \} \cap -W = \{ \mu(v) - \alpha_I - n \beta_I \mid n \geq 0\} \cap -W \neq \emptyset. $$
\epf

\section{Proof of Theorem \ref{thm:gen}}
% We show that, how to compare the two categories. 
So far we have defined the window skeleton $\La_W$ for a given window $W$, and considered the category $Sh^w(\ycal, \La_W)$. 
There is another potentially different definition of window category, which is the CCC image of $\bcal_W$
$$ \acal_W: = \la \{ \tau(\lcal_k), k \in W \} \ra \subset Sh^w(\ycal, \La_\ycal) \ra. $$
There is a fully-faithful embedding  
$$ \iota_W:  \acal_W \into Sh^w(\ycal, \La_W).$$ 
Our goal in this section is to prove the Theorem \ref{thm:gen}, which says, $\iota_W$ is an equivalence if $|W| \geq \eta_-$.

\begin{proof}
Let $W$ be a window, $-W = \{a, a+1, \cdots, b\}$. Assume $|W| \geq \eta_-$. 

For $v \in \Z^N$, let $\fcal(v)$ denote the probe sheaf for point $v + \epsilon \bfone$ in $Sh^w(\R^N, \wt \La_W)$. The collection of probe sheaves $C = \{ \fcal(v) \mid v \in \Z^N\}$ classically generates \footnote{In what follows, we will simply say 'generate' instead of 'classically generate', not to be confused with 'compactly generate'. } $Sh^w(\R^N, \wt \La_W)$.  Our goal is to show that a sub-collection $C(-W)$ also generates, where we use notation
$$ C(J) = \{ \fcal(v) \mid v \in \Z^N, \mu(v) \in J\}, \quad J \subset \Z. $$ 

We will prove in two steps 
\begin{enumerate}
    \item First, we show that $C( (-\infty, b] \cap \Z)$ is generated by $C(-W)$. 
    \item Next, we show that $C$ is generated by $C( (-\infty, b] \cap \Z)$.  
\end{enumerate}

In both steps, we induct on $\mu(v)$. Our induction hypothesis is that, for $\mu(v) \in [k,b]$ where $k \leq a$, $\fcal(v)$ is generated by $C(-W)$.  For the base case, the hypothesis holds for $k=a$. We show if it holds for $k$, then it holds for $k-1$. Let $v \in \Z^N$ with $\mu(v) = k-1$. Consider the microlocal stalk functor for point $(x,\xi) = (v + \epsilon e_{[N]_+}, - e^\vee_{[N]_-})$, it is co-represented by the following chain of probe sheaves
$$ \fcal(v) \to \bigoplus_{J \subset [N]_-, |J|=1} \fcal(v - e_J) \to \cdots \to \bigoplus_{J = [N]_-} \fcal(v - e_J). $$
Since the point $(x,\xi) \notin \wt \La_W$, the above chain complex is acyclic. Hence $\fcal(v)$ is generated by $\fcal(v - e_J)$ for $\emptyset \neq J \subset [N]_-$. The $\mu$-weight of the generators are
$$ \mu(v - e_J) = \mu(v) - \sum_{j \in J} a_i e_j = \mu(v) + \sum_{j \in J} |a_i| e_j $$
Thus $\mu(v - e_J) \in [k, k-1+\eta_-] \subset [k, b]$, hence $\fcal(v)$ is generated by $C([k,b] \cap \Z)$ and in turn is generated by $C(-W)$. 

The second step is similar. Our induction hypothesis is that, for $k \geq b$,  $C( (-\infty, k] \cap \Z)$ is generated by $C( (-\infty, b] \cap \Z)$. It holds for the base case $k=b$. Suppose it holds for $k$, we now consider $\mu(v) = k+1$. By similar argument, we have acyclic chain complex
$$ \fcal(v) \to \bigoplus_{J \subset [N]_+, |J|=1} \fcal(v - e_J) \to \cdots \to \bigoplus_{J = [N]_+} \fcal(v - e_J). $$
Hence $\fcal(v)$ is generated by $ \fcal(v - e_J)$ with $\emptyset \neq J \subset [N]_+$. Since for such $J$, we have $$ \mu(v - e_J) \in \mu(v) + [-\eta_+, -1], $$ $\fcal(v)$ is generated by $C( (-\infty, k]$, hence by $C( (-\infty, b] \cap \Z)$. 
\end{proof}

\bex
Here is an example where $\iota_W$ fails to be an equivalence. Consider $\C^*$ acts on $\C^3$ with weights $(1,1,-2)$, and let $W=\{0\}$ a window of size $1$. The category $\bcal_W$ is equivalent to $\Perf(\C^2/\Z_2)$, where $\C^2/\Z_2$ is the singular affine variety $\spec (\C[x,y,z]/(xy-z^2))$. Then $Sh^w(\ycal, \La_W) \cong \Coh(\C^2/\Z_2)$. Clearly, we have a discrepancy here. 
\eex

\section{Fully-Faithfulness of the Restriction Functor}
Let $W$ be a window, $\La_W \subset T^*Y$ be the window skeleton, and let $\pi: Y \to \R$ be the torus fibration with fiber $\MM_\T$. 
If $I$ is an open interval, we define
$$ \ccal(I) := Sh^\dm_{\La_W}(\pi^{-1}(I)), \ccal^w(I):= Sh^w_{\La_W}(\pi^{-1}(I)).$$
If $I \subset J$ are two intervals, we have restriction functor $\rho_{I,J}$ ,  its left adjoint $\rho_{I,J}^L$ and the restriction $\rho^w_{I,J}$ to compact objects:
$$ \rho_{I,J}: \ccal(J) \to \ccal(I), \quad \rho_{I,J}^L: \ccal(I) \to \ccal(J), \quad \rho_{I,J}^w: \ccal^w(I) \to \ccal^w(J).$$

In this section, we shall study when the restriction functor $\rho_{I,J}$ is fully-faithful. In the next section, we shall study when the co-restriction functor $\rho_{I,J}^L$ (or equivalently $\rho^w_{I,J}$) is fully-faithful.

\subsection{Hom-Sheaf's Singular Support}
We shall use Theorem \ref{thm:SSuhom} to obtain an upper bound. 

First, we define a kernel to create positive isotopy. Since the space is $\R^N$ and the skeleton $\La_N$ is rectilinear, we can use a convolution kernel
$$ K = \C_{ \{ (x,y,t)| t \in (0, 1), x,y \in \R^n, |x_i-y_i| \leq t\}}. $$
If $F \in Sh^\dm(\R^N, \La_N)$, then $K_t^! F = (T_t)_* F$ (the dimension $1$ case is clear), where $T_t: \R^N \to \R^N$ is translation by $-t(1,\cdots,1)$. If $\La \In \La_N$, then let $T_t \La$ denote the positively translated skeleton.

Since the conormal to the boundary of the open set $\pi^{-1}( (a,b))$ are in the direction $\pm a$, hence we are only interested in whether $(x,\pm a) \in \uhom(\wt \La_W, \wt \La_W)$. Since $a$ is not conormal to any coordinate planes of any positive dimension, hence we only need to look at the lattice points. 

Recall the notation of $\La_I$ from section \ref{ss:winskel}. 
\bl
Let $v \in \Z^N$, and $\wt \La_{W,v}$ be the specialization of $\wt \La_W$ at $v$. Then 
\begin{enumerate}
    \item $a \in  \uhom(\wt \La_W, \wt \La_W)_{v}$ only if there exists $I, J \subset [N]$, such that $\La_{I}, \La_{J}$ appears in $\wt \La_{W,v}$ and $I = [N]_-$, $ J \subset [N]_+$.
    \item Dually, $-a \in  \uhom(\wt \La_W, \wt \La_W)_{v}$ only if there exists $I, J \subset [N]$, such that $\La_{I}, \La_{J}$ appears in $\wt \La_{W,v}$ and $I = [N]_+$, $ J \subset [N]_-$.
\end{enumerate} 
\el
\bpf
We only prove the statement about $+a$. Let $\ical_{W,v}$ denote the collection of subset $I\subset [N]$, such that $\La_I \In \wt \La_{W,v}$. Then we have 
$$ \wt \La_{W,v} = \bigcup_{I \in \ical_{W,v}} \La_I. $$
Hence 
\be \label{eq:homhom} \uhom(\wt \La_{W,v}, \wt \La_{W,v}) \subset \bigcup_{I, J \in \ical_{W,v}} \lim_{t \to 0^+} [(\La_I)^a + T_t  \La_J]. \ee
We are interested in which pair $I,J$ can produce a co-vector $a$. Recall that $ T_t \La_J$ is the outward conormal to the open set
$$ T_t P_J  = \{ x \in \R^N: x_i > - t \;\forall i \in J^c \}), $$
and $\La_I^a$ is the inward conormal to 
$$ P_I = \{ x \in \R^N: x_i > 0 \;\forall i \in I^c \}). $$
Hence for $(\La_I)^a + T_t  \La_J$ to contain a covector that has each component nonzero, we need to have $T_t P_J  \cap P_I$ to be a vertex corner, i.e. $I^c \cup J^c = [N]$. The corner $\bfx$ has coordinate $x_i = 0 $ if $i \in I^c$ and $x_i = -t$ if $x \in I \cap J^c$. The fiber of $(\La_I)^a + T_t  \La_J$ over $\bfx$ is then $\sigma_{I^c} - \sigma_I$, hence it can contain $a$ if and only if $I^c = [N]_+$ and $I = [N]_-$. 

Hence, if $a \in  \uhom(\wt \La_W, \wt \La_W)_{v}$, we have $I, J \In \ical_{W,v}$, with $I = [N]_-$ and $J \cap I = \emptyset$, thus $J \subset [N]_+$. 
\epf

We conclude this section with the following theorem. 
\begin{theo} \label{thm:FF} Let $W$ be a window. 
\begin{enumerate}
    \item $(v, +a) \in \uhom(\wt \La_W, \wt \La_W)$ if and only if $v \in \Z^N(W)$ and $\mu(v) + \eta_- \in -W$. 
    \item If $(v, -a) \in \uhom(\wt \La_W, \wt \La_W)$, then $v \in \Z^N(W)$ and $\mu(v) - \eta_+ \in -W$.  
\end{enumerate}
\end{theo}
\bpf
We prove only (1).

{\bf The 'only if' part. } By assumption there exists $I, J \in \ical_{W,v}$ and $I = [N]_-$ $J \subset [N]_+$. If $v \notin \Z^N(W)$, then by Lemma \ref{lm:leftright}, even if we enlarge $W$ so that $|W|$ is large enough, we cannot have both negative-type and positive-type (or emptyset) appears in $\wt \La_{W,v}$, hence $v \in \Z^N(W)$. Then, by Lemma \ref{lm:mix}, we may take $J = \emptyset$, and  by Lemma \ref{lm:mix2} $I = [N]_-$ requires $\mu(v) + \eta_- \in -W$. 

{\bf The 'if' part. } Consider $\fcal = Q_v$, $\gcal = Q_{v - e_{[N]_-}}$, by assumption $SS(\fcal), SS(\gcal) \In \wt \La_W$. Then a direct calculation verifies that
$$ \uhom(\gcal, \fcal) =  (\boxtimes_{i \in [N]_+} \C_{ \{ x \geq v_i\} }) \boxtimes (\boxtimes_{i \in [N]_-} \C_{ \{ x > v_i\} }). $$
Hence the singular support is a product form, and its fiber over $v$ is 
$$ SS(\uhom(\gcal, \fcal))_v = \sigma_{[N]_+} \times (-\sigma_{[N]_-}) $$
Hence $a \in SS(\uhom(\gcal, \fcal))_v \subset \uhom(\wt \La_W, \wt \La_W)_v$.
\epf

\section{Fully-Faithfulness of the Co-restriction Functor}
We continue with the setup of the last section. Our main result is that, 
\bt
If the $|W| \geq \eta_+$, then the co-restriction functor
$$ \rho_{I, J}^L: \ccal(I) \to \ccal(J) $$
is fully-faithful, for any $I \subset J \subset \R$.
\et

More precisely, we have the following proposition. 
\bp \label{pp:SSwHom}
Let $W$ be a window, and $v \in \Z^N$. Then we have
\begin{enumerate}
    \item If $\mu(v) \in -W$, then the co-vectors $(v, \pm a)$ are not in $SS^L_{Hom}(\wt \La_W)$. 
    \item If $\mu(v) > \max(-W)$, and $|W| \geq \eta_+$,  then the co-vectors $(v, \pm a)$ are not in $SS^L_{Hom}(\wt \La_W)$. 
    \item If $\mu(v) < \min(-W)$, and $|W| \geq \eta_-$,  then the co-vectors $(v, \pm a)$ are not in $SS^L_{Hom}(\wt \La_W)$.  
\end{enumerate}
\ep
If any of the above three conditions are satisfied, we say that {\bf the window $W$ is thick enough for $v$}. 

We are going to study the probe sheaves locally near each lattice point. To verify a cotangent vector $(v, a) \in T^* \R^N$ is not in $SS^L_{Hom}(\wt \La_W)$, it suffices to check that $a \notin SS^L_{Hom}(\wt \La_{W,v})_0$, where $\wt \La_{W,v}$ is the specialization of the Lagrangian skeleton $\wt \La_W$ at $v$ and lives in $T^*(T_v\R^N)$.

% As a corollary, we have the following theorem. 
% \bt
% Let $W$ be a window. If $|W| \geq \eta_+$, then for any open intervals (possibly unbounded) $U_1 \subset U_2 \subset \R$, the co-restriction functor 
% $$ Sh^w_{\La_W}(\mu^{-1}(U_1)) \to Sh^w_{\La_W}(\mu^{-1}(U_2))$$
% is fully-faithful, where $\mu: \R^N \qo \MM \to \R$ is the torus fibration. \et

\subsection{Stalk-Probe Sheaves Near a Lattice point}

The skeleton $\wt \La_W$ is contained in $\La_N$. We shall first study the stalk probe sheaf for $\La_{N,v}$. The projection image of $\La_{N,v}$ cut $\R^N$ into $2^N$ different quadrants, which we label as
\[ Q_I = \{ x \in \R^N \mid x_i > 0 \z{ if } i \notin I; x_i < 0 \z { if } i \in I \}, \; \forall I \In [N]. \] 
For example 
\[ Q_\emptyset = (\R_{>0})^N, \quad Q_{[N]} = (\R_{<0})^N. \]
We also define some 'wedge' spaces
\[ P_I = \{ x \in \R^N \mid x_i > 0 \z{ if } i \notin I \}, \; \forall I \In [N]. \] 

Since if $(x,\xi) \in \La_{N,v}$, then $\xi_i \leq 0$, the stalk at $x$ is quasi-isomorphic to the stalk at $x+\epsilon(1,\cdots,1)$ for any small enough positive $\epsilon$. Hence, for any point $x \in \R^N$, we say $x+0 \in Q_I$ if $x+\epsilon(1,\cdots,1) \in Q_I$ for any small enough positive $\epsilon$. For example if $x = 0 \in \R^N$, then $x+0 \in Q_\emptyset$. 

For any fixed $Q_I$, all the point $p$ with $p+0 \in Q_I$ has the same stalk probe sheaves, hence we will talk about stalk probe sheaves for region $Q_I$ instead for individual points. 

\bp \label{pp:fullLaN}
Let $I \subset [N]$. The stalk probe sheaf for region $Q_I$ in category of $Sh^\dm(\R^N, \La_{N,v})$ is $\C_{P_I}$. 
\ep
\bpf
%We will give two proofs, one is algebraic and concise, the other is geometric and might be more intuitive. 

%{\bf Algebraic Proof:} 

The category $Sh^\dm(\R^N, \La_{N,v})$ is equivalent to the module category over the poset of $\{J: J \In [N]\}$, hence the collection of objects $S:= \{\C_{P_I}: I \In [N] \}$ forms a compact generator of $Sh^\dm(\R^N, \La_{N,v})$. Let $p$ be an interior point of $Q_I$.  Since for any $\C_{P_J} \in S$, we have restriction map that is an isomorphism
$$ \Hom(\C_{P_I}, \C_{P_J}) \cong \C_{P_J}|_p \cong \bcs \C & I \In J \\ 0 & \z{ else } \ecs. $$
Hence $\Hom(\C_{P_I},-)$ is the stalk functor at $p$.

\epf

Recall that $\La_I = SS(\C_{P_I})$. 

\bl \label{lm:LaI1}
Let $W$ be a window, $v \in \Z^N$ and $I \subset [N]$. 
If $\La_I \In \wt \La_{W,v}$, then the stalk probe sheaf for $Q_I$ is $\C_{P_I}$. 
\el
\bpf
Since $\wt \La_{W,v} \In \La_{N,v}$, hence $Sh^\dm(\R^N, \wt \La_{W,v})$ embeds fully-faithfully into $Sh^\dm(\R^N, \La_{N,v})$. Hence a stalk probe sheaf $\pcal$ in $Sh^\dm(\R^N, \La_{N,v})$ with $SS(\pcal) \In \wt \La_{W,v}$ is automatically a stalk probe sheaf for $Sh^\dm(\R^N, \wt \La_{W,v})$ for the same region. 
\epf

To see what happens if $\La_I$ does not appear in $\wt \La_W$, we first consider a simpler situation. 
\bl\label{lm:hg}
Consider the following skeleton over $\R^m$
$$ \La_{m,0}' = \bigcup_{\emptyset\neq I \In [m]} SS(\C_{P_I}). $$
Then the stalk probe sheaf for $Q_\emptyset$ in $Sh^\dm(\R^m,  \La_{m,0}' )$ is 
\[ \fcal_{\emptyset} := \uwave{\bigoplus_{|I| = 1} \C_{P_I} } \to \bigoplus_{|I| = 2} \C_{P_I}  \to \cdots \to  \bigoplus_{|I| = m} \C_{P_I} \]
where the underlined entry is at cohomological degree $0$, and the differential is given by the natural inclusion with alternating signs.

Furthermore, the stalks of $ \fcal_{\emptyset}$ are
\[
( \fcal_{\emptyset})_p = \bcs \C & p \in (\R_{>0})^m \\
\C[-m+1] & p \in   (\R_{\leq 0})^m \\
0 & \z{ elsewhere }
\ecs
\]
\el
\bpf %(Algebraic Proof)
Consider the skeleton 
$$ \La_{m,0} = \bigcup_{ I \In [m]} SS(\C_{P_I}). $$
Then $\La_{m,0}  \RM \La_{m,0}'$ is a smooth Lagrangian strata $L$ in the cotangent fiber $T^*_0(R^m)$. 
The microstalk functor for $L$ (unique up to a degree shift) in $Sh^\dm(\R^m, \La_{m,0})$ can be represented as a chain complex of stalk functors $\Phi_I$ for region $Q_I$: 
\[ \Phi_L:= \Phi_\emptyset \to \bigoplus_{|I| = 1} \Phi_I \to  \cdots \to  \bigoplus_{|I| = m} \Phi_I. \]
After restricting to the subcategory $Sh^\dm(\R^m, \La'_{m,0})$, the microstalk functor $\Phi_L$ is zero, hence the chain complex of stalk functors become acyclic.  Thus the co-representing objects  $\fcal_I$ for $\Phi_I$ in $Sh^\dm(\R^m, \La'_{m,0})$ forms an acyclic chain complex:
\[ \fcal_\emptyset \to  \bigoplus_{|I| = 1} \fcal_I \to  \cdots \to  \bigoplus_{|I| = m} \fcal_I \]
After recognizing $\fcal_I = \C_{P_I}$ for $|I|>0$ (c.f. Lemma \ref{lm:LaI1}),  we have the claimed resolution of $\fcal_\emptyset$.  

Next, we verify the claim for the stalks of $\fcal_\emptyset$. If $p \in (\R_{>0})^m$, then the augmented complex 
$$ \C \to [{\bigoplus_{|I| = 1} \C_{P_I}|_p } \to \bigoplus_{|I| = 2} \C_{P_I}|_p  \to \cdots \to  \bigoplus_{|I| = m} \C_{P_I}|_p] $$
become the Koszul complex $\boxtimes_{i=1}^m (\C \to \C)$, hence is acyclic. Thus, the stalk of the complex is $\C$. If $p \in (\R_{\leq 0})^m$, then the only non-zero entry is the last one $\C_{\R^m}|_p$, sitting at degree $m-1$. If $p$ is elsewhere, then the stalk of the complex reduces to a Koszul complex, hence is acyclic.

\epf

\bd
We call the above stalk probe sheaf $\fcal_\emptyset$ on $\R^m$ as the {\bf hourglass sheaf of dimension $m$}, and denoted as $\hourglass_m$. 
\ed

Now, we are ready to continue to describe the stalk probe sheaf for $\wt \La_W$ near $v$. 

Let $[N] = \{1,\cdots, N\}$, and the subsets $[N]_+ = \{i: a_i > 0\}$ and $[N]_- = \{i: a_i < 0\}$. For any subset $I \subset [N]$, we denote $I_\pm = I \cap [N]_\pm$. We denote $I_\pm^c = [N]_\pm \RM I_\pm$. We decompose $\R^N$ as 
$$ \R^N = \R^{N_+} \times \R^{N_-}. $$
If $\fcal_+ \in Sh( \R^{N_+} )$ and $\fcal_- \in Sh( \R^{N_-} )$, we will write $\fcal_+ \boxtimes \fcal_-$ for the product sheaf in $\R^N$. If particular, if $J_+ \subset [N]_+$, we will write $P_{J_+} \in Sh(\R^{N_+})$.

\bp \label{pp:posneg}
Let $v \in \Z^N$, $I \subset [N]$. Assume the window $W$ is thick enough for $v$, and $\La_I$ does not appear in $\wt \La_W$. Then exactly one of following cases happen \ref{lm:leftright})
\begin{enumerate}
    \item $I \subset [N]_+$ and for any $\emptyset\neq K \In [N]_-$, $\La_{I \sqcup K} \In \wt\La_{W,v}$. 
    \item $I \subset [N]_-$ and for any $\emptyset\neq K \In [N]_+$, $\La_{I \sqcup K} \In \wt\La_{W,v}$. 
\end{enumerate}

\ep
In case (1) (resp. (2)) , we say $(v,I)$ is {\bf of positive (resp. negative) type}. 

\bpf
We enumerate all possible cases of $(v,I)$. If $v \in \Z^N(W)$, and $\La_I$ does not appear, then $I \neq \emptyset$ and $I \subset [N]_+$ or $I \subset [N]_-$. Suppose $I \In [N]_+$, then for any $\emptyset\neq K \In [N]_-$, $I \sqcup K$ is of mixed type, hence $\La_{I \sqcup K}$ appears in $\wt\La_{W,v}$ by Lemma \ref{lm:mix}. 

If $v \notin \Z^N(W)$, and if $\mu(v) > \max(-W)$, then $I = \emptyset$ or $I$ is of negative type, by Lemma \ref{lm:leftright}. We claim that we are in case (2), sicne for any $\emptyset\neq K \In [N]_+$, $I \sqcup K$ is of mixed type, hence $\La_{I \sqcup K}$ appears in $\wt\La_{W,v}$. We cannot be in case (1), since even if $I=\emptyset$ then adding a non-empty negative type subset $K$ still makes $I \sqcup K$ a negative type, hence $\La_{I \sqcup K} \not \In \wt\La_{W,v}$.

The case for $v \notin \Z^N(W)$ and $\mu(v) < \min(-W)$ is similar. We claim that we are in case (1) and omit the proof. 
\epf

\bp \label{pp:prod}
Let $W$ be a window,  $v \in \Z^N$ and $I \subset [N]$. Assume the window $W$ is thick enough for $v$, and $\La_I$ does not appear in $\wt \La_{W,v}$. Then  
\begin{enumerate}
    \item If $(v, I)$ is of positive type, then the stalk probe sheaf for region $Q_I$ is
\[ \fcal_{I} = \bigoplus_{|J_-| = 1} \fcal_J \to \bigoplus_{|J_-| = 2} \fcal_J  \to \cdots  \to \bigoplus_{|J_-| = N_-} \fcal_J \cong \C_{P_{J_+}} \boxtimes \hourglass_{N_-},  \]
where  $J = J_+ \sqcup J_-$, and $J_+ = I_+=I$. 
    \item If $(v, I)$ is of negative type, then the stalk probe sheaf for region $Q_I$ is
\[ \fcal_{I} = \bigoplus_{|J_+| = 1} \fcal_J \to \bigoplus_{|J_+| = 2} \fcal_J  \to \cdots  \to \bigoplus_{|J_+| = N_+} \fcal_J \cong \hourglass_{N_+} \boxtimes \C_{P_{J_-}},  \]
where  $ J = J_+ \sqcup J_-,$ and $J_- = I_- = I$. 
\end{enumerate}
\ep
\bpf
We only prove  (1) since the proof of (2) is exactly the same. 
Let $p \in Q_I$ with $p_i = +1$ if $i \notin I$ and $p_i = -1$ if $i \in I$. Consider the affine subspace $M = p + \R^{N_-}$.  The restriction of $\La_W$ on $M$ is the same as in Lemma \ref{lm:hg}, with $m=N_-$. Hence, we have an acyclic chain complex of stalk functors
\[ \Phi_I \to \bigoplus_{|J_-| = 1} \Phi_J \to \bigoplus_{|J_-| = 2} \Phi_J  \to \cdots  \to \bigoplus_{|J_-| = N_-} \Phi_J \]
where $J$ has the same positive part $J_+$ as $I$, and $J_-$ runs through all subsets of $[N]_-$. Hence the stalk probe sheaves forms the similar chain complex, replacing $\Phi_J$ by $\fcal_J$. Thanks for the condition that $(v,I)$ is of positive type, we see for $|J_-| \geq 1$, $\La_J = \La_{I \sqcup J_-}$ appears in $\wt \La_{W,v}$, and 
$$ \fcal_J = \C_{P_J} = \C_{P_{J_+}} \boxtimes \C_{P_{J_-}}, \; \forall |J_-| \geq 1. $$ 
Thus, we may pull out the common factor $\C_{P_{J_+}}$ and get $\fcal_I \cong \C_{P_{J_+}} \boxtimes \hourglass_{N_-}$.
\epf

\begin{corollary}
Let $v \in \Z^N$ and $W$ is a window thick enough for $v$. For any $I \In [N]$, let $\fcal_I$ be the probe for region $Q_I$. 
The category $Sh^w(\R^N, \wt \La_{W,v})$ is split-generated by $\fcal_I$ for those $I$ such that $\La_I$ appears in $\wt \La_{W,v}$.
\end{corollary}
\begin{proof}
Since $Sh^w(\R^N, \wt \La_{W,v})$ is split-generated by all $\{\fcal_I : I \subset [N]\}$, and any $\fcal_I$ can be written as a chain complex using only $\{\fcal_I : \La_I \subset \wt \La_{W,v} \}$, hence the smaller set of $\fcal_I$ split-generates. 
\end{proof}

\subsection{Co-restriction of stalk-probe sheaves near a lattice point}

Let $W$ be a window,  $v \in \Z^N$. 
We define half-spaces in $T_v \R^N \cong \R^N$
$$ B_\pm = \{ x \in \R^N: \pm x \cdot a > 0 \}. $$
Let $\fcal_I$ denote the stalk probe sheaf for region $Q_I$.

As a warm-up, we consider the full skeleton's specialization $\La_{N,v}$. 
\bl
The stalk probe sheaf for $Q_I \cap B_\pm$ in $Sh^\dm(B_\pm, \La_{N,v})$ is $\C_{P_I}|_{B_\pm}$, if $Q_I \cap B_\pm \neq \emptyset$. 
\el
\bpf
Consider the restriction to $B_-$ only, since the other case is similar. $Q_I \cap B_- = \emptyset$ if and only if $I = [N]_-$. The category $Sh^\dm(B_\pm, \La_{N,v})$ is equivalent to the module category over the poset $\{ I \In [N]: I \neq [N]_- \}$, hence $\{\C_{P_I}|_{B_-} \mid  I \In [N],  I \neq [N]_- \}$ compactly generate the category. The rest of the proof is the same as 
Proposition \ref{pp:fullLaN}.
\epf

\bp \label{pp:main}
Assume $W$ is thick enough for $v$. Then the stalk probe sheaf for $Q_I \cap B_\pm$ is $\fcal_I|_{B_\pm}$, if $Q_I \cap B_\pm \neq \emptyset$. 
\ep
\bpf
Consider the restriction to $B_-$ only. If $\La_I$ appears in $\wt \La_{W,v}$ and $I \neq [N]_-$, then $\fcal_I = \C_{P_I}$, $\fcal_I |_{B_-}$ is the probe for region $Q_I \cap B_-$ in $Sh^\dm(B_\pm, \La_{N,v})$, hence is also the probe for the same region in $Sh^\dm(B_\pm, \wt \La_{W,v})$. 

If $\La_I \not \In \wt \La_{W,v}$, then we consider the following two cases. 
\begin{enumerate}
    \item $(v,I)$ is of positive type, i.e $I \In [N]_+$ and $I \sqcup K \in \ical_{W,v}$ for any non-empty $K \In [N]_-$. 
    \begin{enumerate}
        \item Assume $I \neq \emptyset$.
    Consider $p \in Q_I$ of the following form, $p = - e_I + \epsilon e_{I^c}$. 
    Then $\mu(p) = - \sum_{i \in I} a_i + O(\epsilon)$, hence for small enough $\epsilon$, $\mu(p)<0$ and $p \in Q_I \cap B_-$. And the slice $ S = p + \R^{N_-}$ contains the point $p'$ where $p'_i = p_i$ for $i \in [N]_+$ and $p'_i = 0$ for $i \in [N]_-$. There are $2^{N_-}$ strata on the slice $S$ near $p'$, they belong to $Q_J$ with $J_+ = I$ and $J_- \In [N]_-$ respectively. Hence, we have an acyclic chain complex of stalk functors, resulting in an acyclic chain complex of probe sheaves 
    $$ \fcal_{Q_I \cap B_-} \to  \bigoplus_{|J_-| = 1} \C_{P_J \cap B_-} \to \bigoplus_{|J_-| = 2} \C_{P_J \cap B_-}  \to \cdots  \to \bigoplus_{|J_-| = N_-} \C_{P_J \cap B_-} $$
    where we plugged in the probe sheaf for $Q_J \cap B_-$ by $\C_{P_J \cap B_-}$ since $J \in \ical_{W,v}$.  This acylic complex can be viewed as a resolution of the first term $\fcal_{Q_I \cap B_-}$ , and the rest of the chain complex is $\fcal_I|_{B_-}$, hence we get
    $$ \fcal_{Q_I \cap B_-} \cong \fcal_I|_{B_-}.$$
    \item Assume $I = \emptyset$, this is only possible if $v \notin \Z^N(W)$, and more precisely we have $\mu(v) < \min(-W)$. We claim
    $$ SS(\C_{P_\emptyset}|_{B_-}) \In \bigcup_{i \in [N]_-} \La_{\{i\}}|_{B_-} \In \wt \La_{W,v}, $$
    Indeed, if $x \in \pa(P_\emptyset) \cap B_-$, then if we define $K = \{i: x_i \neq 0\}$, then $K \not\In [N]_+$, and 
    $SS(\C_{P_\emptyset})|_x \In \La_J|_x$ if and only if $K^c \In J^c$, i.e. $J \In K$. Hence we may choose $J = \{j\}$ for some $j \in K \cap [N]_-$. This verfies the claim. Finally, we see
    $$ \C_{P_\emptyset}|_{B_-} = [\C_{\R_+^{N_+}} \boxtimes \C_{\R_+^{N_-}} ] |_{B_-}=  [\C_{\R_+^{N_+}} \boxtimes \hourglass_{{N_-}} ] |_{B_-} = \fcal_\emptyset|_{B_-} $$
    Since $B_-$ does not intersects with the product $\R_+^{N_+} \times \R_{\leq 0}^{N_-}$, hence $\hourglass_{{N_-}}$ acts as $\C_{\R_+^{N_-}}$.
    \end{enumerate}
    \item $(v,I)$ is of negative type, i.e $I \In [N]_-$ and $I \sqcup K \in \ical_{W,v}$ for any non-empty $K \In [N]_+$. Since $Q_I \cap B_- \neq \emptyset$, $I \neq [N]_-$. 
      \begin{enumerate}
        \item Assume $I \neq \emptyset$.
    Consider $p \in Q_I$ of the following form, $p = - \epsilon e_I + e_{[N]_- \RM I} +  \epsilon e_{[N]_+}$. 
    Then $\mu(p) = \sum_{i \in [N]_- \RM I} a_i + O(\epsilon)$, hence for small enough $\epsilon$, $\mu(p)<0$ and $p \in Q_I \cap B_-$. And the slice $ S = p + \R^{N_+}$ contains the point $p'$ where $p'_i = p_i$ for $i \in [N]_-$ and $p'_i = 0$ for $i \in [N]_+$. There are $2^{N_+}$ strata on the slice $S$ near $p'$, they belong to $Q_J$ with $J_- = I$ and $J_+ \In [N]_+$ respectively. The rest of the argument is similar to case $(v,I)$ of positive type, $I\neq \emptyset$, and we omit the details. 
    \item The case $I = \emptyset$ is impossible by the proof in Proposition \ref{pp:posneg}. 
    \end{enumerate}
\end{enumerate}

\epf

\ss{Proof of Proposition \ref{pp:SSwHom}}
Let $W$ be a window, $v \in \Z^N$, assume $W$ is thick enough for $v$.
Consider the restrictions
$$ \rho_{v,\pm a}: Sh^\dm(T_v \R^N, \wt \La_{W,v}) \to Sh^{\dm}(B_\mp , \wt \La_{W,v}), \quad B_\pm =  \{x \in T_v \R^N: \pm x \cdot a > 0 \}$$
and the left-adjoint $\rho^L_{v,\pm a}$. We are going to show that $id \to \rho_{v,\pm a} \circ \rho^L_{v,\pm a}$ are isomorphisms.

This is precisely what Proposition \ref{pp:main} is about. Consider the co-restriction from $B_-$ to $\R^N$. Suppose $Q_I \cap B_- \neq \emptyset$, then by Proposition \ref{pp:main},  $\fcal_I|_{B_-}$ is the probe sheaf for $Sh^\dm(B_-, \wt \La_{W,v})$ for region $Q_I \cap B_-$. Tautologically, we have 
$$ \rho^L_{v, a} (\fcal_I|_{B_-}) = \fcal_I, $$
hence 
$$ \rho_{v, a} \circ \rho^L_{v, a} (\fcal_I|_{B_-}) = \fcal_I|_{B_-}. $$
This holds for all the compact generators of $Sh^\dm(B_-, \wt \La_{W,v})$, hence $\rho^L_{v, a}$ is an isomorphism. Similarly for $\rho^L_{v, a}$.

\section{Proof of Theorem \ref{thm:main}}
In previous two sections, we studied the under what conditions the restriction and the co-restriction functors are fully-faithful: roughly speaking, if $|W|$ is too large, then the restrction functor will not be fully-faithful; if $|W|$ is too small, then the co-restriction functor will not be fully-faithful.

In this section, we will let $|W| = \eta_+$, the minimum width of the window to ensure the co-restriction functor to be fully-faithful. To be concrete, let 
$$ -W = \{ a, a+1, \cdots, b\}, \quad b = a + \eta_+ - 1. $$
Since $SS^L_{Hom}=\emptyset$, from Proposition \ref{p:2cond} and Lemma \ref{lm:sshom2}, we have
$$ SS(\wt \La_W) = \uhom(\wt \La_W, \wt \La_W). $$
From Theorem \ref{thm:FF}, we see $(v, a) \in SS(\wt \La_W)$ if $v \in \Z^N$ and 
$\mu(v) \in [a, a+\eta)$, and $(v,-a)$ is not in $SS(\wt \La_W)$ for any $v$.  

Let $[v] = p(v) \in \R^N \qo \MM$. The above statement also holds if we replace $v$ by $[v]$ and $\wt \La_W$ by $\La_W$. Recall from section 5 that $\ccal(I) = Sh^\dm(\pi^{-1}(I), \La_W)$. 
\bp\label{p:sublevel}
\begin{enumerate}
    \item As $t$ increases from $-\infty$ to $+\infty$, $\ccal( (-\infty, t))$ is locally constant except when $t = a, a+1, \cdots,a+\eta_1$. 
    $$ \ccal( (-\infty, t)) \cong \bcs  Sh^\dm(\MM_\T, \La_-) &  \z{ if } t \ll 0 \\ 
     Sh^\dm(\MM_\T, \La_+) & \z{ if } t \gg 0
     \ecs $$
    \item As $t$ decreases from $+\infty$ to $-\infty$, $\ccal( (t, +\infty))$ remains invariant, and equals $ Sh^\dm(\MM_\T, \La_+)$. 
\end{enumerate} 
\ep
\begin{proof}
This are consequences of the description of the singular support $SS(\La_W)$. 
\end{proof}

Let $\ccal(<t)$ denote $\ccal((-\infty, t))$. 
\bp
For any $t \in \R$,  we have
$$ \ccal(<t) \cong   \ccal(t):=Sh^\dm(\MM_\T, \La_W|_t) $$
\ep
\begin{proof}
(1) If $t \notin \{a, a+1, \cdots, a+\eta-1)\}$,
then for any $\epsilon$ small enough,  we have 
$$ \ccal(<t) \cong \ccal( (t-\epsilon, t+\epsilon)). $$
Since $\ccal( (t-\epsilon, t+\epsilon)) \cong \ccal(I)$ for any open $I \subset (t-\epsilon, t+\epsilon)$,  hence the restriction $\ccal( (t-\epsilon, t+\epsilon))$ to $\ccal(t)$ is an equivalence. 

(2) If $t \in \{a, a+1, \cdots, a+\eta-1)\}$, we claim that $\ccal(t) \cong \ccal(t-\epsilon)$ for some small $\epsilon$. Given the claim, we then have $\ccal(t) \cong \ccal(t-\epsilon) \cong \ccal(<t-\epsilon) \cong \ccal(<t)$. Now we prove the claim, let $[v]$ denote the vertex in slice $\pi^{-1}(t)$, then the stratification of $\ycal_t$ compared with $Y_{t-\epsilon}$ has one region less. If we lift $[v]$ to $v$ and work in the tangent space of $v$, then the region $Q_{[N]_+}$ does not intersect with the slice $B_0$ where $B_\delta = \{x \cdot a = \delta\}$. However, $\La_{[N]_+}$ does not appear in $\wt \La_{W,v}$, hence the stalk in region $Q_{[N]_+}$ can be deduced from the stalks in other regions, $Q_I$ where $I$ is of of mixed type. Thus $\ccal(t) \cong \ccal(t-\epsilon)$ and the claim holds. 

% Let $\chi: (-1, 1) \to (-1,1)$ be a weakly monotone surjective smooth map, such that $\chi([-1/2,1/2]) = 0$, and $\chi'(x)> 0$ for $x \in (-1, -1/2) \cup (1/2,1)$. Then, we have map 
% \[ \phi:  M \times (-1,1) \to M \times (t-\epsilon, t+\epsilon), \quad (x,s) \mapsto (x, t + \epsilon \chi(s)). \]
% Then we have skeleton $\La'_W:= \phi^* \La_W|_{(t-\epsilon, t+\epsilon)}$

\end{proof}

Now we consider how $\ccal((-\infty, t))$ jump as $t$ cross one of the critical moments. 
\bp
Let $t \in \{a, a+1, \cdots, a+\eta-1)\} $, then we have semi-orthogonal decomposition 
$$ \ccal((-\infty, t+\epsilon)) = \la \Vect, \ccal((-\infty, t-\epsilon)) \ra $$
for small enough $\epsilon$.
\ep
\bpf

The co-restriction functor $\rho^L: \ccal(< t-\epsilon) \to  \ccal(< t+\epsilon)$ is fully-faithful, and let $S$ denote the image of $\rho^L$. WE claim that the right-orthogonal $S^\perp$ is generated by one simple object,  thus $S^\perp \cong \Vect$ and we get the desired semi-orthogonal decomposition. Consider the Lagrangian $\wt \La_{W,v}$, for $v \in \Z^N(W)$. By assumption, $\mu(v) + \eta_- \in -W$, hence for any $I \In [N]_-$,  $\La_{I} in \wt \La_{W,v}$. Thus, consider the object in $Sh^\dm(\R^N, \wt \La_{W,v})$
$$ \C_{P_\emptyset} \to {\bigoplus_{|I| = 1} \C_{P_I} } \to \bigoplus_{|I| = 2} \C_{P_I}  \to \cdots \to  \uwave{\bigoplus_{|I| = N_-} \C_{P_I}} $$
where $I \In [N]_-$ and the underlined cell is in degree $0$. We see it is actually $\C_{\R_{>0}^{N_+}} \boxtimes \C_{\R_{\leq 0}^{N_-}}$, hence is not supported on $B_-$ and is a simple object. 
\epf

By repeat this decomposition $\eta$ times, and using Proposition \ref{p:sublevel}, we get 
\begin{corollary}
$$ Sh^\dm(\MM_T, \La_+) \cong \ccal(a+\eta) \cong \la \Vect, \ccal(a+\eta-1) \ra  \cong \cdots \cong  \la \underbrace{\Vect,\cdots, \Vect}_{\eta \z{ times }}, Sh^\dm(\MM_T, \La_-)\ra. $$
\end{corollary}

We can be more precise about the relations of these $\Vect$s. For $t \in \{a, a+1, \cdots, a+\eta-1)\}$, and $v \in \mu_\Z^{-1}(t)$, we may consider the chain complex 
\[ \pcal_I(v) : = Q_v \to \bigoplus_{J \subset I, |J|=1} Q_{v - e_J} \to \bigoplus_{J \subset I, |J|=2} Q_{v - e_J} \to \cdots \to Q_{v - e_I}, \]
where $I = [N]_-$. Then $SS(\pcal_I(v)) \in \wt \La_W$. And $\pcal_I(v)$ is the microstalk skyscraper sheaf for the Lagrangian smooth strata $(v + (0,1)^{I^c}) \times (-\sigma_I)$ over the corresponding face of the cube $v + [0,1]^N$. These collections of $p_! \pcal_I(v)$ corresponds to the $\Vect$ factors, and under CCC, they go to structure sheaf for the unstable loci of the $(-)$ GIT quotient, with appropriate $\C^*$-weights. We omit the detail of the computation here, since it is of the same structure as the following B-model computation. 

\subsection{Window subcategory in B-model \label{ss:b-model}}
In this section, we describe semi-orthogonal decomposition in B-model. We follow closely \cite{segal2011equivalences} and \cite[Theorem 2]{kerr2017homological}. 

Let $V = \C^N$, and we decompose $V = V_+ \oplus V_-$ according to the weight of the $\C^*$ action. Then $V_+$ is also the unstable locus for $\chi_-$. We show that $Coh(\chi_+)$ admits semi-orthogonal decomposition with component $\ocal_{V_+}$ and $Coh(\chi_-)$. 

Recall that $\ocal_{V_+}$ admits a locally free resolution, the Koszul resolution. For concreteness, we assume $V_+$ is cut-out by equation $z_1=\cdots=z_m=0$, and let $\ocal$ denote the structure sheaf of $V$, then  
$$ \ocal \cdot (z_1 \wedge \cdots z_m) \to \cdots \to \bigoplus_{i<j} \ocal \cdot (z_i \wedge z_j) \to \bigoplus_{i} \ocal \cdot z_i \to \uwave{\ocal} \cong \ocal_{V_+} $$
where the wavy underline indicate the term is of homological degree zero. The homological degree $-k$ term can be identified with $\ocal \otimes \wedge^k V_-^\vee$. One can also lift the resolution to the $\C^*$-equivariant category, if the weight of the degree $0$ term $\ocal$ is $0$, then the weight of the degree $-m$ term is $\eta_-$, since $V_-^\vee$ are have only positive weights. We denote such an equivariant lift as $\ocal_{V_+}([0, \eta_-])$, and sometimes also as $\ocal_{V_+}(0)$, where the interval indicate the ranges of the weights in the Koszul resolutions.

\begin{lemma}
In the unequivariant category, we have the following hom spaces, 
\begin{enumerate}
    \item $\Hom(\ocal_{V_+}, \ocal_{V}) \cong Sym^*( V_+^\vee ) \otimes \wedge^{\z{top}} V_-  $
    \item $\Hom(\ocal_{V}, \ocal_{V_+}) = Sym^*(  V_+^\vee )$
    \item $\End(\ocal_{V_+}, \ocal_{V_+}) = Sym^*( V_+^\vee) \otimes \wedge^* V_- $
\end{enumerate}
where $\wedge^i V_- := Sym^i(V_-[-1])$. 
\end{lemma}
\bpf
It suffices to check the one-dimensional case. 
% If $\C^*$ acts on $V = \C$ with weight $a$, then we have 
% \begin{align*} 
% \Hom(\ocal_V(0), \ocal_V(0)) &= Sym^*(V^\vee)  \\
% \Hom(\ocal_0, \ocal_\C(0)) &= \C[-1](a) \\
% \Hom(\ocal_\C(0), \ocal_0(0)) &= \C(0) \\
% \Hom(\ocal_0(0), \ocal_0(0)) &= \C(0) \oplus \C[-1](a) 
% \end{align*}
\epf

\begin{corollary}
In the $\C^*$-equivariant category, we have
$$ \Hom_\CS(\ocal_{V}(i), \ocal_{V_+}(j)) = Sym^*( V_+^\vee)_{i-j} 
$$
and 
$$ \Hom_\CS(\ocal_{V_+}(i), \ocal_{V_+}(j)) = Sym^*( V_+^\vee \oplus V_-[-1])_{i-j} 
$$
In particular, if $i>j$, then both hom vanishes. 
\end{corollary}
\begin{proof}
Recall that our $\C^*$-equivariant hom convention at in section \ref{ss:notation}
$$ \Hom_\CS(F(i), G(j)) = ( \Hom (F, G) \cdot e_{j-i} )^{\C^*} $$
where $e_{j-i}$ has weight $j-i$. Hence one need to take the weight $i-j$ part of $\Hom (F, G)$. Next, $V_+^\vee$ and $V_-$ have only negative weights, hence their symmetric powers also only have negative weights. 
\end{proof}

Suppose $W = \{0, \cdots, \eta_+-1\}$ Then, we have semi-orthogonal decomposition of 
$$ \bcal_W = \Coh(\xcal_+) \cong  \la \underbrace{\ocal_{V_+}(0), \ocal_{V_+}(1), \cdots, \ocal_{V_+}(\eta -1)}_{\eta \z{ many} }, \underbrace{\la \{ \ocal_{V}(\eta),  \ocal_{V}(\eta+1), \cdots,  \ocal_{V}(\eta_+-1)\} \ra}_{ \Coh(\xcal_-)} \ra. $$
Indeed, by the previous corollary, this is a semi-orthogonal decomposition. 

We also record the following results, where the window size is more general. %This will be useful when we consider the perverse schober setting, where we have $\eta_+=\eta_-$ and the window size is $\eta_+ + 1$. 

\bp
Let $W=\{a,a+1,\cdots, b-1,  b\}$ be a window of size $|W|$.  Let $\C^*$ acts on $V = \C^N$ and decompose $V = V_+ \oplus V_-$ according to the weights. Let $\eta_\pm$ be the absolute values of weights of $\det(V_\pm)$. Then
\begin{enumerate}
     \item If $|W| \geq \eta_-$, then we have semi-orthogonal decomposition, 
     $$ \bcal_W = \la \ocal_{V_+}(a), \ocal_{V_+}(a+1), \cdots,  \ocal_{V_+}(b-\eta_-), \underbrace{\la \{ \ocal_V(b-\eta_- + 1), \cdots,  \ocal_V(b) \} \ra}_{Coh(\chi_-)} \ra. $$
    \item If $|W| \geq \eta_+$, then we have semi-orthogonal decomposition
    $$ \bcal_W = \la \ocal_{V_-}(b), \ocal_{V_-}(b-1), \cdots,  \ocal_{V_-}(a+\eta_+), \underbrace{\la \{ \ocal_V(a+\eta_+-1), \cdots, \ocal_V(a) \} \ra}_{Coh(\chi_+)} \ra. $$

\end{enumerate}
\ep
\begin{proof}
The proof for (1) is the same as above. For, we may consider a $\C^*$ action on $V$ with opposite weights, then the result follows from (1). 
\end{proof}

\bibliographystyle{amsalpha}
\bibliography{HMS}

\end{document}